\documentclass[a4paper,twoside,english]{article}
\usepackage[dvips]{graphicx}
\usepackage{color}
\definecolor{blu}{rgb}{0.0,0.0,1.0}

\usepackage{amsfonts}
\usepackage{listings}          %listati
\usepackage{fancyhdr}
\usepackage{indentfirst}

\usepackage{array} 
\usepackage{newlfont}
\usepackage{microtype} % migliora il riempimento delle righe
\usepackage[english]{babel}
\usepackage[applemac]{inputenc}    %ATTENZIONE A QUESTO COMANDO MAC (di solito compare latin1
\usepackage[T1]{fontenc}
\usepackage{amscd}

\usepackage[fleqn,tbtags]{amsmath}
\usepackage{amssymb,mathrsfs}       %per le formule matematiche fuori dal testo 
\usepackage{dsfont}                            %\mathds{1}
\usepackage{latexsym}
\usepackage{amsthm}
\usepackage{fancyhdr}

\theoremstyle{definition}                          %stile corsivo
\newtheorem{thm}{Theorem}[section]     %definizione ambiente teorema
\newtheorem{prop}[thm]{Proposition}
\newtheorem{Hypothesis}[thm]{Hypothesis}%definizione ambiente proposizione
\newtheorem{lem}[thm]{Lemma}             %definizione ambiente lemma
\theoremstyle{definition}               %stile roman
\newtheorem{dfn}[thm]{Definition}%  definizione ambiente definizione
\newtheorem{ese}[thm]{Example}        %definizione ambiente esempio
\newtheorem{rem}[thm]{Remark}

\newtheoremstyle{prova}% hnamei
{3pt}% Space above
{3pt}% Space below
{}% Body font
{}% Indent amounti
{\textbf}% Theorem head font
{.}% hPunctuation after theorem headi
{\newline}% hSpace after theorem headi2
{}% Theorem head spec (can be left empty, meaning `normal')
\theoremstyle{prova}

\def\A{\mathbb A}

\def\R{\mathbb R}
\def\N{\mathbb N}
\def\M{\mathbb M}
\def\Z{\mathbb Z}

\def\D{\mathbb D}
\def\E{\mathbb E}

\def\P{\mathbb P}
\def\Q{\mathbb Q}
\def\R{\mathbb R}
\def\S{\mathbb S}
\def\U{\mathbb U}
\def\V{\mathbb V}
\def\X{\mathbb X}
\def\Y{\mathbb Y}
\def\K{\mathbb K}

\def\shb{{\cal B}}
\def\shc{{\cal C}}
\def\shd{{\cal D}}
\def\shf{{\cal F}}
\def\shg{{\cal G}}
\def\shm{{\cal M}}

\def\shl{{\cal L}}
\def\shp{{\cal P}}

\def\1{\mathds{1}}
\def\X{\mathbb{X}}
\def\Y{\mathbb{Y}}
\def\Z{\mathbb{Z}}
\def\M{\mathbb{M}}
\def\W{\mathbb{W}}
\def\V{\mathbb{V}}

\def\be{\begin{equation}}
\def\ee{\end{equation}}

\begin{document}
% \author{{\sc Giorgio Fabbri} \footnote{Centre d'Etudes des Politiques
%  Economiques de l'Universit\`e d'Evry, Evry (France)}}
%  %{\sc,}\ 
% \author{{\sc Cristina Di Girolami}\footnote{Laboratoire Manceau de Math\'ematiques, Facult\'e des Sciences et Techniques, Universit\'e du Maine, D\'epartement de Math\'ematiques, Avenue Olivier Messiaen, 72085 Le Mans CEDEX 9 (France)}
%      {\sc,}\  
% \author{{\sc Giorgio Fabbri} \footnote{Centre d'Etudes des Politiques
%  Economiques de l'Universit\`e d'Evry, Evry (France)}}
%  %{\sc,}\ 

% \ {\sc and}\ {\sc Francesco RUSSO} %$^*$}
% \footnote{Unit\'e de Math\'ematiques appliqu\'ees,
% 828, boulevard des Mar\'echaux,
% F-91120 Palaiseau (France).}
%}

\author{Cristina Di Girolami \footnote{Università G.D'Annunzio di Pescara, Dipartimento di Economia Aziendale (Italy) and Laboratoire Manceau de 
Math\'ematiques, Facult\'e des Sciences et Techniques,
 Universit\'e du Maine, D\'epartement de Math\'ematiques, Avenue Olivier Messiaen, 72085 Le Mans Cedex 9 (France).
E-mail: c.digirolami@unich.it}, \;
Giorgio Fabbri\footnote {EPEE, Universit\'e
d'Evry-Val-d'Essonne (TEPP, FR-CNRS 3126), D\'epartement d'Economie, 4
Bd. Fran\c{c}ois Mitterrand, 91025 Evry cedex (France). E-mail: giorgio.fabbri@univ-evry.fr} \; and \;  Francesco Russo\footnote{ENSTA ParisTech, Unit\'e de Math\'ematiques appliqu\'ees, 
828, boulevard des Mar\'echaux,
F-91120 Palaiseau   (France).
E-mail: francesco.russo@ensta-paristech.fr} }

\date{July 16th 2013}
\title{The covariation for  Banach space valued processes and applications}
%\begin{document}
\maketitle

{\bf Abstract}

This article focuses on a recent concept of 
 covariation  for processes taking values in a separable
Banach space $B$ and a corresponding 
 quadratic variation.
The latter is more general than the classical one of M\'etivier
and Pellaumail. 
 Those  notions are associated with some  subspace $\chi$ 
of the dual of the projective tensor product of $B$ with itself.
 We also introduce the notion of a
convolution type process, which is a natural generalization of
the It\^o process and the concept of $\bar \nu_0$-semimartingale,
which is a natural extension of the classical notion of
semimartingale. The framework is the 
stochastic calculus via regularization in Banach spaces.
 Two main applications  are mentioned: one related
to Clark-Ocone formula for finite quadratic variation processes;
the second one concerns the probabilistic representation
of a Hilbert valued partial differential equation of Kolmogorov type.

\medskip

[\textbf{2010 Math Subject Classification}: \ ] \ {60G22, 60H05, 60H07,  60H15, 60H30, 26E20, 35K90 46G05}
\medskip


%\tableofcontents

\section{Introduction and motivations}
		\label{sec:intr}

The notion of covariation and quadratic variation are fundamental in stochastic
calculus related to Brownian motion and semimartingales.
However, they also play a role in stochastic calculus
for non-semimartingales.

In the whole paper a fixed strictly positive time $T>0$ 
will be fixed.
 Given a real continuous process $X = (X_t)_{t \in [0,T]}$, 
there are two classical definitions of quadratic variation related to it,
denoted by $[X]$.
The first one, inspired to \cite{fo}, says that, when it exists,
$[X]_t$ is a continuous process being the limit, in probability,  
  of $\sum_{i=0}^{n-1} (X_{t_{i+1} \wedge t} - X_{t_{i} \wedge t})^2 $ where
$0 = t_0 < t_1 < \ldots < t_n  = T$ is element of a sequence of
subdivisions whose mesh $\max_{i = 1}^{n-1} (t_{i+1} - t_i)$
 converges to zero. 
The second one, less known,
% according to \cite{Rus05}, 
is based on stochastic calculus 
via regularization; it characterizes $[X]$ as the continuous process 
such that $[X]_t$ is the limit in
 probability  for every $t \in [0,T]$, when $\varepsilon
\rightarrow 0$, 
of $ \frac{1}{\varepsilon} \int_0^t (X_{(s+\varepsilon) \wedge T} - X_s)^2 ds, t \in [0,T]$.
In all the known examples both definitions give the same result.
We will use here the  second formulation, which looks operational and simple.
% we refer to \cite{Rus05}
%for a survey on that notion and stochastic regularization.
If $[X]$ exists then $X$ is called {\bf finite quadratic variation}
process. A real process $X$ such that $[X] \equiv 0$ is called 
{\bf zero quadratic variation process}; we also say in this case that
$X$ has a zero quadratic variation.
If $X$ is a (continuous) semimartingale, $[X]$ is the classical bracket.
Consequently, if $W$ is the real Brownian motion
then $[W]_t = t$.

Our generalization of quadratic variation of  processes
taking values in a Banach space $B$,
called $\chi$-quadratic variation, is recalled at 
 Section \ref{SecChi6}; 
see \cite{DGR1, DGR2, DGR} for more exhaustive information.
 This has 
significant infinite dimensional applications but it also has
 motivations in the study of real  stochastic processes with finite 
quadratic variation,
 even for Brownian motion and semimartingales.

Indeed, the class of real finite quadratic  variation
processes is quite rich even if many important fractional type processes
do not have this property.
Below we enumerate a list of such (continuous) processes.
A survey of stochastic calculus via regularization which focuses 
on covariation is \cite{Rus05}.
\begin{enumerate}
\item A bounded variation process has zero quadratic variation.
\item  A semimartingale with
 decomposition $S=M+V$, $M$ being  a local martingale and 
 $V$ a bounded variation process is a finite quadratic variation process
with $[S]=[M]$.  
\item A fractional Brownian motion $X = B^H, 0 < H < 1$ has finite 
quadratic variation if and only if $ H \ge \frac{1}{2} $.
If $  H > \frac{1}{2} $, it is a zero quadratic variation process.
\item 
An important subclass of finite quadratic variation processes
is constituted by  {\bf Dirichlet} processes, which should more
properly be called {\bf F\"ollmer-Dirichlet}, since they were introduced
by H. F\"ollmer \cite{FolDir}; they  were later 
further investigated by J. Bertoin,
see \cite{ber}. 
An a $({\mathscr F}_{t})$-\textbf{Dirichlet} process admits a (unique) decomposition
of the form $X = M+A$, where $M$ is an  $({\mathscr F}_{t})$-local martingale
and $A$ is a zero quadratic variation (such that $A_0 = 0$ a.s.).
In this case $[X] = [M]$. 
It is simple to produce Dirichlet processes $X$
with the same quadratic variation as Brownian motion. 
Consider for instance $X = W + A$ where $W$ is a classical Brownian motion and 
$A$ has zero quadratic variation. In general we postulate that
$A_0 = 0$ a.s. so that the mentioned decomposition is unique.
\item Another interesting example is the bifractional Brownian
motion, introduced first by \cite{hv}. Such a process $X$ 
%= B^{H,K}$
 depends on two parameters $0 < H < 1,  0 < K \le 1$ and it is often denoted by
 $B^{H,K}$.
If $ HK > \frac{1}{2} $ then   $B^{H,K}$ has zero quadratic variation.
 If
$K = 1$, that process is a fractional Brownian motion with parameter $H$.
A singular situation produces when $HK = \frac{1}{2}$.
In that case $X$ is a finite quadratic variation  process 
and $[X]_t = 2^{1-K} t$. That process is neither a semimartingale nor
a Dirichlet process, see \cite{rtudor}, see Section 3.1.
In particular not all the finite quadratic variation processes
are Dirichlet processes.
\end{enumerate}

A simple link between real  and Banach
space  valued processes is the following.
Let $ 0  < \tau \le T$. Let
$X  = (X_t, t \in [0,T])$  be a  real continuous process, that 
we naturally prolong in the sequel for $t \le 0$ setting $X_t = X_0$ 
and $X_t = X_T$ if $t \ge T$.
 The process $X(\cdot)$ defined by  
$ \X = X(\cdot)=\{X_{t}(u):=X_{t+u}; u \in [-\tau,0]\},$ 
constitutes the $\tau$-memory of  process $X$.
The natural state space for $\X$ is the non-reflexive 
separable space $B = C([-\tau,0])$. 
$X(\cdot)$ is 
the so called {\it window}   process associated with $X$
(of width $\tau > 0$).
If $X$ is a Brownian motion (resp. semimartingale, diffusion, Dirichlet
 process), then $X(\cdot)$ will be called window Brownian motion
(resp. window semimartingale, window diffusion, window Dirichlet process).

If $X = W$ is a classical Wiener process, $\X = X(\cdot)$ has no natural
quadratic variation, in the sense of Dinculeanu or M\'etivier and
Pellaumail,  see Subsection \ref{window}. However it will possess
 a more general quadratic variation called $\chi$-quadratic variation,
 which is related to  specific sub-Banach spaces
 $\chi$ of $(B \hat \otimes_\pi B)^\ast$.
\bigskip

A first  natural application of our covariational calculus is
motivated as follows. 
 If $h\in L^{2}(\Omega)$, the martingale representation theorem
 states the existence of a predictable 
process $\xi\in L^{2}(\Omega\times [0,T])$ such that 
$h=\mathbb{E}[h]+\int_{0}^{T}\xi_{s}dW_{s}$. If $h\in \mathbb{D}^{1,2}$ in
 the sense of Malliavin calculus, see for instance  \cite{nualartSEd, 
malliavin}, the celebrated {\bf Clark-Ocone} formula 
says  $\xi_{s}=\mathbb{E}\left[ D^{m}_s  h \vert {\mathscr F}_{s}\right] $
where $ D^{m}$ is the Malliavin gradient.
So 
\begin{equation}		\label{eq F11}
h=\mathbb{E}[h]+\int_{0}^{T}\mathbb{E}\left[ D^{m}_s h 
\vert {\mathscr F}_{s}\right] dW_{s}.
\end{equation}
A.S. Ustunel \cite{ustunelC-O}
obtains a generalization of \eqref{eq F11} when 
$h \in L^{2}(\Omega)$,
making use of the predictable projections of a Wiener distributions 
in the sense of S. Watanabe \cite{waTaInst}.\\
A natural question is the following: is Clark-Ocone formula
{\it robust} if       the law of $X=W$ is not anymore the
 Wiener measure but $X$ is still a finite quadratic variation process
 even not necessarily a semimartingale?
Is there a reasonable class of random variables $h$ for which
a representation of the type
$ h = H_{0}+  '' \int_{0}^{T} \xi_{s} dX_{s} ''$, $H_0 \in \R$, $\xi$ adapted?
Since $X$ is a not a semimartingale, previous integral
has of course to be suitably defined, in the spirit of a limit
of Riemann-Stieltjes non-anticipating sum. 
We have decided however to interpret the mentioned integral as a forward
integral in the regularization method, see Section \ref{finreg}.
We will denote it by $\int_{0}^{T} \xi_{s} d^-X_{s}$. \\
So let us suppose that  $X_{0}=0$, $[X]_{t}=t$ and
 $\tau = T$ for simplicity.
We look for a 
 reasonably rich   class of functionals 
 $ G: C([-T,0]) \longrightarrow \R $ such that the r.v. 
$ h:=G(X_{T}(\cdot))$    %\quad \mbox{Borel functional}
admits a representation of the type
%\begin{block}{}
\begin{equation}\label{I1}
 h = G_{0} + \int_{0}^{T}\xi_{s}d^-X_{s}, 
\end{equation}
provided that
 $G_{0}\in \R$ and $\xi$ is an adapted process with respect to the canonical filtration of $X$.
The idea is to express
 $h=G(X_{T}(\cdot))$ as $ u(T, X_{T}(\cdot))$ or in some cases
\[
h=G(X_{T}(\cdot)) = \lim_{t\uparrow T}u(t, X_{t}(\cdot)),
\]
where $u\in C^{1,2}\left([0,T[\times C([-T,0])  \right)$ solves 
an infinite dimensional 
partial differential equation,
and \eqref{I1} holds 
with $\xi_t = D u(t,\eta)(\{0\}), t \in ]0,T[$. 
At this point we will have 
$ h = u(0,X_{0}(\cdot))+\int_{0}^{T} \xi_s
d^{-}X_{s} $,
recalling that 
$ D\, u:[0,T]\times C([-T,0])\longrightarrow (C([-T,0])^\ast
 = \mathcal{M}([-T,0])$. This is the object of Section \ref{SIDPDE}.
A first step in this direction was done in \cite{DGRnote} and more in details in Chapter 9 of \cite{DGR}. 
\bigskip

A second interesting application concerns convolution processes, see Section \ref{Conv}. Consider $H$ and $U$ two
 separable Hilbert spaces and a $C_0$-semigroup $(e^{tA})$ on $H$, see Sections \ref{SGPF} for
 definitions and references. Let $\W$ be an $U$-values $Q$-Wiener process for some positive bounded operator $Q$ on
$U$. Let $\sigma =  (\sigma_t, t \in [0,T])$  %with paths a.s. in  $\shl^2(H,K)$ 
and $b =  (b_t, t \in [0,T])$ %with paths taking values in $H$ 
two suitable predictable integrands,
see Section \ref{DBZI} for details.
An $H$-valued convolution process has  the following form:
\begin{equation} \label{FEDPS}
 \X_t = e^{tA} x_0 + \int_0^t e^{(t-r)A} \sigma_r d \W_r + \int_0^t e^{(t-r)A} b_r dr,
t \in [0,T],
\end{equation}
for some $x_0 \in H$. Convolution type processes are an extension
of It\^o processes, which appear when $A$ vanishes. 
Mild solutions of infinite dimensional evolution equations are in natural way convolution processes. 
They have no scalar quadratic variation even if  driven by a one-dimensional Brownian motion. Still it can be proved that they admit a ${\chi}$-quadratic variation for some suitable space $\chi$, see Proposition \ref{PConv}.\\
Another general concept of processes that we will introduce   is the one of
 $\bar \nu_0$-semimartingales. An $H$ valued process $\X$ is said
$\bar \nu_0$-semimartingale if there is Banach space $\bar \nu_0$ including $H$
(or in which $H$ is continuously injected) so that $\X$
is the sum of an $H$-valued local martingale and a bounded variation  $\bar \nu_0$-valued
process. A convolution process will be shown to
be a $\bar \nu_0$-semimartingale, where the dual ${\bar \nu_0}^\ast $ equals $D(A^\ast)$,
 see item 2.
of Proposition \ref{PConv}.

Let us come back for a moment to real valued processes.
A real process $X$ is called {\bf weak Dirichlet} (with respect to a given filtration), 
 if it can be written as the sum of a local martingale and a process $A$ such that $\left[ A,N\right] =0$ 
for every continuous local martingale.
A significant result of F. Gozzi and F. Russo, see \cite{rg2}, is the following.
If $f:[0,T] \times \R \rightarrow \R$ is of class $C^{0,1}$, 
then $Y_t = f(t,X_t), t \in [0,T]$ is a weak Dirichlet process.
A similar result, in infinite dimension, is obtained replacing the process $X$ with its associated window $X(\cdot)$.
%see \cite{DGR2}.
The notion of Dirichlet process extends  to the infinite dimensional
 framework via the notion of 
{\it $\nu$-weak Dirichlet process}, see Definition \ref{def:nu-dir}.
 An interesting example of $\nu$-weak Dirichlet process is given, once more, by convolution processes, 
see Proposition \ref{PConv}.

Generalizing that result of \cite{rg2}, it can be proved that, given
 $u:[0,T] \times \R \rightarrow \R$
 of class $C^{0,1}$ and being  $\X$ a suitable
 $\nu$-weak Dirichlet process with finite $\chi$-quadratic variation,
where $\chi$ is Chi-subspace associated with $\nu$,
 then $Y_t = u(t, \X_t)$ is a real weak Dirichlet process.
 Moreover its (Fukushima-Dirichlet type) 
 decomposition is provided in Theorem \ref{th:prop6}. That theorem can be seen as a substitution-tool
 of It\^o's formula if $u$ is not smooth and is a key tool for the application 
we provide in Section \ref{sec:Kolmogorov}.
Examples of such $\nu$-weak Dirichlet processes are convolution type
processes, or more generally $\bar \nu_0$-semimartingales, see Proposition
\ref{P67}, item 2.
In Section \ref{sec:Kolmogorov}, we study the solution
 of a non-homogeneous
 Kolmogorov equation and we provide a uniqueness result for the related solution. The proof of that result is based on  a 
representation 
 for (strong) solutions of 
the Kolmogorov equation that is obtained thanks to the uniqueness of the 
 decomposition of a real weak Dirichlet process. The uniqueness result covers
 cases that, as far as we know, were not yet included in the
 literature.
For instance,  in our results, the initial datum $g$ of the Kolmogorov
 equation is asked   be continuous but we do not require
 any boundedness assumption on it.
 This kind of problem cannot be studied if the problem is approached, 
as in \cite{CerraiGozzi95, Gozzi97}, looking at the properties of 
the transition semigroup on the space $C_b(H)$ (resp. on $B_b(H)$)
 of continuous and bounded (resp. bounded) functions defined on $H$, because,
 in this case, the initial datum always needs to be bounded.
 More details are contained in Section \ref{sec:Kolmogorov}.
In the same spirit, further applications to stochastic verification theorems, 
in which the Kolmogorov type equation, is replaced by an Hamilton-Jacobi-Bellman equation, can be realized, see for instance \cite{RusFab}.

\section{Preliminaries} \label{SPrelim}

\subsection{Functional analysis background}
 \label{SPrelim1}

In the whole paper, $B$ (resp. $H$) will stand  for a separable Banach
(resp. Hilbert) space.
$\vert \cdot \vert_B$ (resp. $\vert \cdot \vert_H$)
 will generally denote the norm related to $B$ (resp. $H$).
However, if the considered norm appears clearly, we will often only
 indicate it by
$\vert \cdot \vert$. Even the associated inner product with 
$\vert \cdot \vert_H$ will be indicated by $\langle \cdot, \cdot \rangle_H$
or simply by $\langle \cdot, \cdot \rangle$.

Given an element $a$ of a Hilbert space $H$, we generally denote
by $a^\ast$, the corresponding element of $H^\ast$ via Riesz identification.
We will use the identity  ${}_{H^\ast}\langle a^\ast, b \rangle_H
=  \langle a, b \rangle_H =  \langle a, b \rangle_H  $ without comments.
Let $B_1, B_2$ be two separable real Banach spaces.
 We denote by $B_1\otimes B_2$ the algebraic tensor product defined
 as the set of the elements of the form
 $\sum_{i=1}^n x_i\otimes y_i$, for some positive integer $n$
 where $x_i$ and  $y_i$ are respectively elements of $B_1$ and $B_2$.
 The product  $\otimes\colon B_1\times B_2 \to B_1\otimes B_2$ is bilinear.

A natural norm on  $B_1\otimes B_2$ is the projective norm $\pi$: for all $u\in B_1\otimes B_2$, we denote
 by $\pi(u)$ the norm
\[
\pi(u) := \inf \left \{ \sum_{i=1}^n  |x_i|_{B_1} |y_i|_{B_2} \; : \; 
u = \sum_{i=1}^{n}  x_i \otimes y_i  \right \}.
\]
This belongs to the class of the so-called {\it reasonable norms} $\vert \cdot \vert$, in particular 
verifying  $\vert x_1 \otimes x_2 \vert = \vert x_1 \vert_{B_1} \vert x_2 \vert_{B_2}  $, if $x_1 \in B_1, x_2 \in B_2$.
We denote by $B_1\hat\otimes_\pi B_2$ the Banach space obtained as completion of $B_1\otimes B_2$
 for the norm $\pi$, see \cite{Ryan02} Section 2.1.
We remark that its topological dual $(B_1\hat\otimes_\pi B_2)^*$ is isomorphic to the space 
of continuous bilinear forms ${\mathcal Bi}(B_1,B_2)$ of
 continuous bilinear forms,
 equipped with the norm $\Vert \cdot  \Vert_{B_1,B_2}$ where 
$\Vert \Phi \Vert_{B_1,B_2} = \sup_{\begin{subarray}{l}
                                     a_1 \in B_1, a_2 \in B_2 \\
                                     |a_1|_{B_1}, |a_2|_{B_2}\leq 1,
                                     \end{subarray}
}|\Phi(a_1, a_2)|$.

\begin{lem}
\label{lm:aggiunto}
Let $B_1$ and $B_2$ be two separable, reflexive real Banach spaces. 
Given $a^* \in B_1^*$ and  $b^* \in B_2^*$ we
 can associate to $a^*\otimes b^*$ the elements $j(a^* \otimes b^*)$ of $(B_1\otimes B_2)^*$ acting 
as follows on a generic element $u = \sum_{i=1}^{n}  x_i \otimes y_i
 \in B_1\otimes B_2$:
\[
\left\langle j(a^*\otimes b^*), u \right \rangle = \sum_i^n \left\langle a^*,
 x_i \right\rangle \left\langle b^*, y_i \right\rangle.
\]
$j(a^*\otimes b^*)$ extends by continuity to the whole $B_1\otimes B_2$ 
and its norm in $(B_1\otimes B_2)^*$
 equals $|a^*|_{B_1^*}|b^*|_{B_2^*}$. In particular if $\nu_i$ is a (dense) 
subspaces of $B_i^\ast, i = 1, 2,  $ then
the projective tensor product $\nu_1 \hat \otimes_\pi \nu_2$  can be seen
 as a subspace of
 $(B_1 \hat \otimes_\pi B_2)^\ast$.
\end{lem}
\begin{proof}
See \cite{RusFab} Lemma 2.4.
\end{proof}
\begin{rem} \label{RSuccessivo}
We remark that $B_1\hat{\otimes}_{\pi}B_2$ fails to be Hilbert even if 
$B_1$ and $B_2$ are  Hilbert spaces. It is not even  reflexive space.
Fore more information about tensor topologies, we refer e.g. to 
\cite{Ryan02}.
\end{rem}

Let us consider now two separable Banach spaces ${B_1}$ and ${B_2}$. 
With $C(B_1; B_2)$, we symbolize the set of the locally bounded 
continuous $B_2$-valued functions defined on $B_1$. This is a Fr\'echet 
type space with the seminorms
\begin{equation} \label{NormC}
\|u\|_r := \sup \left \{ \vert u(x) \vert_{B_2} \; : \; x\in B_1, \; with \, |x|_{B_1} \leq r \right \}
\end{equation}
for $r\in \mathbb{N^\ast}$.

If
 ${B_2}=\mathbb{R}$ we will often simply use the notation $C({B_1})$
 instead of $C({B_1};\mathbb{R})$. Similarly, given a real interval $I$,
typically $I = [0,T]$ or $I = [0,T[$, 
%for a real number $T>0$ we 
we use the notation $C(I \times {B_1}; {B_2})$ for the set of the
 continuous ${B_2}$-valued
 functions defined on $I\times {B_1}$ while we use the lighter notation 
$C(I \times {B_1})$ when ${B_2}=\mathbb{R}$. 
%This is a Fr\'echet type space equipped with the topology of 
%the uniform convergence on compact sets.
For a function $u: I \times B_1 \rightarrow \R, \ 
(t,\eta) \mapsto u(t,\eta) $, we denote by
$(t,\eta) \mapsto Du(t,\eta) $ (resp. $(t,\eta) \mapsto D^2 u(t,\eta) $)
(if it exists) the first (resp. second) Fr\'echet 
derivative  w.r.t. the variable $\eta \in {B_1}$).
Eventually a function $
(t,\eta) \mapsto u(t,\eta)
 \in C(I \times {B_1})$  (resp. $u  \in C^1(I \times {B_1})$) 
  will be said to belong to  $C^{0,1}(I \times {B_1})$
 (resp. $C^{1,2}(I \times {B_1})$)
 if $Du$  exists and it is continuous,
i.e. it belongs to $C(I \times {B_1};{B_1}^*)$ 
 (resp.  $D^2 u (t,\eta$) exists for any $(t,\eta) \in
[0,T] \times B_1$ and it is continuous, i.e.  it belongs
to $C(I \times {B_1}; {\mathcal Bi}({B_1},{B_1}))$.\\
By convention all the continuous functions defined
on an interval $I$ are naturally extended 
by continuity to $\R$.

We denote 
by $\mathcal{L}({B_1}; {B_2})$ the space of linear bounded maps from 
${B_1}$ to ${B_2}$.
It is of course a Banach space and  it is a topological 
subspace of   $C({B_1}; {B_2})$;   we will denote by
$\Vert \cdot \Vert_{\mathcal{L}({B_1}; {B_2})}$ the corresponding norm.
We will often indicate in the sequel by a double bar, i.e.
 $\Vert \cdot \Vert$, the norm of an operator or more 
generally   the seminorm of a function.
As a particular case, if we denote by $U, H$ two separable Hilbert spaces,
 $\shl(U;H)$ will be the space of linear bounded maps
from $U$ to $H$. If $U= H$, we set $\shl(U) := \shl(U;U)$.
$\shl_2(U;H)$ will be the set of
 {\it Hilbert-Schmidt} operators from $U$ to $H$ and $\shl_1(H)$
 (resp. $\shl_1^+(H)$) will be the space of (non-negative) 
{\it nuclear} operators on $H$.
For  details about the notions of Hilbert-Schmidt and nuclear operator, the reader may consult \cite{Ryan02}, Section 2.6 and \cite{DaPratoZabczyk92} Appendix C.
If $ T \in \shl_2(U;H)$ and $T^\ast:H\rightarrow U$ is the adjoint operator, then $T T^\ast \in \shl_1(H)$ and the Hilbert-Schmidt
norm of $T$ gives
$
\Vert T \Vert^2_{\shl_2(U;H)} = \Vert T T ^\ast \Vert_{\shl_1(H)}.$
We recall that, for a generic element $T\in {\mathcal L}_1(H)$
 and given 
 a basis $\left \{ e_n \right \}$ of $H$ the sum
$
\sum_{n=1}^{\infty} \left \langle T e_n, e_n  \right \rangle
$
is absolutely convergent and independent of the chosen basis 
$\left \{ e_n \right \}$. It is called \emph{trace} of $T$
 and denoted by ${\mathrm Tr}(T)$. 
$\shl_1(H)$ is a Banach space and we denote by $\Vert \cdot \Vert_{\shl_1(H)}$
the corresponding norm.
If $T$ is non-negative then ${\mathrm Tr}(T) = \| T \|_{{\mathcal L}_1 (H)}$ and
 in general we have the inequalities
% $|Tr(T)| \leq \| T \|_{{\mathcal L}_1 (H)}$ and, see Proposition C.1,
%\cite{DaPratoZabczyk92}.
\begin{equation}
\label{eq:trace-h1}
|{\mathrm Tr}(T)| \leq \| T \|_{{\mathcal L}_1 (H)}, \quad
\sum_{n=1}^{\infty} |\left \langle T e_n, e_n  \right \rangle| \leq \| T
 \|_{{\mathcal L}_1 (H)},
\end{equation}
see Proposition C.1,
\cite{DaPratoZabczyk92}.
As a consequence,  if $T$ is a non-negative operator,  the  relation below
\begin{equation} \label{EHSQ}
\Vert T \Vert_{\shl_2(U;H)}^2 = {\mathrm Tr} (T T^\ast),
\end{equation}
will be very useful in the sequel.

Observe that every  element $u \in H\hat\otimes_\pi H$  is isometrically associated 
with an element $T_u$ in the space of nuclear operators  
$\shl_1(H)$.
The identification (which is in fact an isometric isomorphism) associates to any element $u$
 of the form $\sum_{i=1}^{\infty} a_n\otimes b_n$ in  $H\hat\otimes_\pi H$ the nuclear operator $T_u$ defined as 
 \begin{equation}\label{Isomorphism}
 T_u(x) := \sum_{i=1}^{\infty} \left\langle x, a_n \right\rangle b_n,
 \end{equation}
see for instance \cite{Ryan02}, Corollary 4.8 Section 4.1 page 76. 

We recall that, to each element $\varphi $ of $(H\hat\otimes_\pi H)^*$,
 we can associate  a bilinear continuous map
 $B_\varphi$ 
and a linear continuous operator 
$L_\varphi: H \rightarrow H$ 
 such that
\begin{equation}
\label{eq:expressionLB}
\left\langle L_\varphi (x), y \right\rangle = B_\varphi(x,y) = 
\varphi(x\otimes y),\qquad \text{for all $x, y \in H$},
\end{equation}
see \cite{Ryan02},  the discussion before Proposition 2.11 
Section 2.2. at page 24.
One can prove  the following, see \cite{RusFab},
 Proposition 2.6 or \cite{DGR}, Proposition 6.6.

\begin{prop}
\label{prop:26}
 Let  $u \in  H\hat\otimes_\pi H$  and $\varphi \in (H\hat\otimes_\pi H)^*$
with associated maps
 $T_u \in  \mathcal{L}_1(H), L_\varphi \in   \mathcal{L}(H;H)$.
Then
\[
{}_{(H\hat\otimes_\pi H)^*} 
\langle \varphi,u \rangle_{H\hat\otimes_\pi H} = Tr \left ( T_u L_\varphi \right ).
\]
\end{prop}
\begin{prop} \label{RExchange}
Let $g: [0,T] \mapsto \shl_1^+(H)$ measurable such that 
\begin{equation}
\label{eq:inrgh1}
\int_0^T \| g(r) \|_{\shl_1(H)} dr <\infty. 
\end{equation}
Then $\int_0^T g(r) dr \in \shl_1^+(H)$ and its trace equals
 $\int_0^T {\mathrm Tr}(g(r)) dr$.
\end{prop}
\begin{proof} 
$\int_0^T g(r) dr \in \shl_1(H)$ by the the first inequality of
\eqref{eq:trace-h1} and by Bochner integrability property. Clearly
the mentioned integral is a non-negative operator.
The remainder
 follows quickly from the relation between the trace and the
 $\shl_1(H)$ norm that we have recalled above; indeed if $(e_n)$ is an orthonormal
basis,
\[
\sum_{n=1}^{N} \left \langle \int_0^T g(r)\,  d r \, e_n, e_n  \right \rangle 
 = \int_0^T \sum_{n=1}^{N} \left \langle g(r) e_n, e_n  \right \rangle d r
\]
and we can pass to the limit thanks to (\ref{eq:trace-h1}),
 (\ref{eq:inrgh1}) and Lebesgue's dominated convergence theorem.
\end{proof}

\subsection{General probabilistic framework}

\label{SGPF}

In the whole paper we will fix $T>0$. $(\Omega, \shf, \P)$
will be a fixed probability space and $\shp$ will denote the
predictable $\sigma$-field on $\Omega \times  [0,T]$.  
$(\mathscr F_t) = (\mathscr F_t, {t \in [0,T]})$ will  be a filtration fulfilling the usual 
conditions. If $B$ is a  Banach space,  $\shb(B)$
will denote its Borel $\sigma$-algebra. 
A $B$-valued random variable $C$ is integrable if $\E(\vert C \vert)$ is finite
and the quantity $\E(C)$ exists as an element in $B$.
It fulfills in particular the Pettis property: $ \varphi(\E(C)) = \E(\varphi(C))$ for any $\varphi\in B^*$.

Given a $\sigma$-algebra $\shg$, the random element
$\E(C \vert \shg): \Omega \rightarrow B$
denotes the conditional expectation of $C$ with respect to $\shg$.
The concept of conditional expectation for $B$-valued random elements, when 
$B$ is a separable Banach space, are recalled for instance in \cite{DaPratoZabczyk92} Section 1.3.
In particular,  for every $\varphi \in B^\ast$ we have
$\E\left(\Psi {}_{B^*} \langle \varphi, C \rangle_B \right) = \E\left( \Psi 
{}_{B^*} \langle \varphi, \E ( C  \vert \shg)  \rangle_B  \right)$,
for any bounded r.v. $\shg$-measurable $\Psi$.

 A  stochastic process
will stand for an application 
$[0,T] \times \Omega \rightarrow B$, which is measurable 
with respect to the $\sigma$-fields $\shb([0,T]) \otimes \shf$
and $\shb(B)$.
%In that case, it is always strongly (Bochner) measurable,
%since it is the  the limit of $\shf$-measurable countably-valued functions. 
If $B$ is infinite dimensional, the processes are indicated by the bold letters
$\X, \Y, \Z$. 
Given a Banach space $B_0$, a process  $\X: ([0,T]\times \Omega, \shb([0,T]) \otimes \shf) 
\to B$ is said to
 be {\bf strongly (Bochner) measurable} if it is the limit of $\shf$-
measurable countably-valued functions. 
If $B_0$ is separable then any measurable  process is always strongly  measurable,
since it is the  the limit of $\shf$-measurable countably-valued functions. 
A reference about measurability for functions taking values in Banach spaces
is for instance \cite{vanNervUMD}, Proposition 2.1.
A process $[0,T] \times \Omega \rightarrow B$, which
is measurable with respect to the $\sigma$-fields $\shp$ and
$\shb(\R)$ is said to be {\bf predictable} with respect to the given 
filtration  $(\mathscr{F}_{t}, t \in [0,T])$.

 Let  $H, U$ be separable Hilbert spaces, $Q \in \mathcal{L}(U)$ be a positive,
 self-adjoint operator and define $U_0:=Q^{1/2} (U)$. This  is again a separable Hilbert space.
Even if not necessary we suppose $Q$ to be injective, which avoids formal complications. We 
 endow $U_0$ with the scalar product $\left\langle a,b \right\rangle_{U_0} := \left\langle Q^{-1/2}a,Q^{-1/2}b \right\rangle$. 
$Q^{1/2} \colon U \to U_0$ is an  isometry, see e.g. \cite{DaPratoZabczyk92} 
Section 4.3.
Assume that $\W^Q=\{\W^Q_t:0\leq t\leq T\}$ is an $U$-valued 
$(\mathscr{F}_t)$-$Q$-Wiener process (with $\W^Q_0=0$, $\mathbb{P}$ a.s.).
 The notion
of $Q$-Wiener process and
 $(\mathscr F_t)$-$Q$-Wiener process were defined for example in \cite{DaPratoZabczyk92} Chapter 4, see
also \cite{GawareckiMandrekar10} Chapter 2.1.
We recall that  $\mathcal{L}_2(U_0; H)$ stands for the Hilbert space of the Hilbert-Schmidt operators from $U_0$ to $H$.

An $U$-valued process $\M \colon [0,T] \times \Omega \to U$ 
 is called  $(\mathscr{F}_t)$-{\bf martingale} if, for all 
$t\in [0,T]$, $\M$ is $(\mathscr{F}_t)$-adapted
with $\mathbb{E} \left [ |\M_t|_U \right ] < +\infty$ and $\mathbb{E}\left [ \M_{s}  |\mathscr{F}_{t} \right ] = \M_{t}$ 
for all $0\leq t \leq s \leq T$.
In the sequel, the reference to the filtration  $(\mathscr{F}_{t}, t \in [0,T])$
will be often omitted.
 The mention ``adapted'', 
''predictable'' etc... we will always refer to {\it with respect to
 the filtration} $\left \{ \mathscr{F}_t \right \}_{t\geq 0}$.
An $U$-valued martingale $\M$ is said to be {\bf square integrable} if 
 $\mathbb{E} \left [ |\M_T|^2_U \right ] < +\infty$.
A $Q$-Wiener process is a square integrable martingale.
We denote by $\mathcal{M}^2(0,T; U)$ the linear space of square integrable  
 martingales indexed by $[0,T]$ with values
in $U$, i.e.   of measurable processes $\M: [0,T] \times \Omega 
 \rightarrow U $ such that $E(\vert \M_T \vert^2_U) < \infty$. 
In particular for $\M \in \mathcal{M}^2(0,T; U)$, the quantity
\[
|\M|_{\mathcal{M}^2(0,T; U)} := \left 
(\mathbb{E} \sup_{t\in [0,T]} |\M_t|_U^2 \right )^{1/2}
\]
is finite. Moreover, it defines a norm and $\mathcal{M}^2(0,T; U)$ 
endowed with it, is a Banach space as
 stated in \cite{DaPratoZabczyk92}% page 79, 
Proposition 3.9.
An $U$-valued process $\M \colon [0,T] \times \Omega \to U$  is 
called {\bf local martingale} if there exists a non-decreasing sequence 
of stopping times $\tau_n\colon \Omega \to [0,T]\cup\{+\infty\}$ 
such that $\M_{t\wedge \tau_n}$ for  $t\in[0,T]$ is a martingale and $\mathbb{P} \left [ \lim_{n\to\infty} \tau_n = +\infty \right ]=1$. 
All the considered martingales and local martingales 
will be supposed to be continuous. 

Given a continuous local martingale $\M \colon [0,T] \times \Omega \to U$,
 the process $|\M|^2$ is a real local sub-martingale,
 see Theorem 2.11 in \cite{KrylovRozovskii07}. 
The increasing predictable process, vanishing at zero, appearing in the Doob-Meyer decomposition of $|\M|^2$
will be denoted by $([\M]^{\mathbb{R}, cl}_t, t\in [0,T])$. It  is of course uniquely determined and continuous.

A $B$-valued process $\A$ is said to {\bf be a bounded variation process}
or {\bf to have bounded variation} if almost every trajectory has 
bounded variation i.e. if, 
for almost all $\omega$, 
the supremum of 
$\sum_{i=1}^N |\A_{t_{i-1}}(\omega) - \A_{t_i}(\omega)|_B$
over all the possible subdivisions $ 0 = t_0 < \ldots < t_N$, $N \in \N^\ast$,
is finite. If $B = U$ is a Hilbert space,
 following \cite{Metivier82}, Definition 23.7, we  say that an 
$U$-valued process $\X$ is a {\bf semimartingale} if $\X$ can be written 
as $\X = \M + \A$ where $\M$ is a local martingale and $\A$ a
 bounded variation process. The total variation function process
associated with $\A$ is defined similarly as for real valued processes
and it is denoted by $t \mapsto \Vert \A_t \Vert$.

\subsection{The Hilbert space valued It\^o stochastic integral}
\label{DBZI}

We recall here some basic facts about the Hilbert space valued It\^o integral, which was made popular for instance by G. Da Prato
 and J. Zabczyk, see \cite{DaPratoZabczyk92, DaPratoZabczyk96}.
 More recent monographs on the subject are \cite{GawareckiMandrekar10, PrevotRockner07}.

Let  $H$ and $U$ be two separable Hilbert spaces. We adopt  the notations that we have introduced in previous subsection \ref{SGPF}.
 $\mathcal{I}_\M(0,T;U,H)$ will be the set of the processes $\X\colon [0,T]\times \Omega \to \mathcal{L}(U;H)$ 
that are strongly measurable from $([0,T]\times \Omega, \mathscr{P})$ to $\mathcal{L}(U;H)$ and such that
\[
|\X|_{\mathcal{I}_\M(0,T;U,H)} := \left (\mathbb{E} \int_0^T \| \X_r \|_{\mathcal{L}(U;H)}^2 d [\M]^{\mathbb{R}, cl}_r \right )^{1/2} < +\infty.
\]
$\mathcal{I}_\M(0,T;U,H)$ endowed with the norm $|\cdot|_{\mathcal{I}_\M(0,T;U,H)}$ 
is a Banach space.
The linear map
\[
\left \{
\begin{array}{l}
I\colon \mathcal{I}_\M(0,T;U,H) \to \mathcal{M}^2(0,T; H)\\
\X \mapsto \int_0^T \X_r d \M_r,
\end{array}
\right .
\]
is a contraction, see e.g. \cite{Metivier82} Section 20.4 above Theorem 20.5. 
As illustrated in  \cite{KrylovRozovskii07} Section 2.2 (above Theorem 2.14),
the stochastic integral w.r.t. $\M$ extends to the integrands $\X$ which are 
strongly 
measurable from $([0,T]\times \Omega, \mathscr{P})$ to $\mathcal{L}(U;H)$ and such that
\begin{equation} \label{EChainRule}
\int_0^T \| \X_r \|_{\mathcal{L}(U;H)}^2 d [\M]^{\mathbb{R}, cl}_r < +\infty \qquad a.s.
\end{equation}
We denote by $\mathcal{J}^2(0,T; U,H)$ such a 
family of integrands w.r.t. $\M$.

We have the following standard fact, see e.g. \cite{KrylovRozovskii07} Theorem 2.14. 
\begin{prop} \label{PChainRule-1}
Let $\M$ be a continuous $U$-valued  $({\mathscr F}_t)$-local martingale,
$\X$ a process verifying (\ref{EChainRule}). Then $\N_t = \int_0^t \X_r d\M_r, t \in [0,T],$ is an  $(\mathscr{F}_t)$-local martingale with values in $H$.
\end{prop}

Consider now the case when the integrator $\M$ is a $Q$-Wiener process,
with values in $U$, where
$Q$ be again a positive injective  and self-adjoint operator  
$Q \in \mathcal{L}(U)$,
see Section \ref{SGPF}. 
We  consider $U_0$ with its inner product as before.
 By \eqref{EHSQ} we can easily prove that, given 
$A \in \mathcal{L}_2(U_0; H)$, we have
 $\| A \|^2_{\mathcal{L}_2(U_0; H)}  = {\mathrm Tr} \left(A Q^{1/2}(A Q^{1/2})^*
 \right).$
 Let $\W^Q=\{\W^Q_t:0\leq t\leq T\}$ be an $U$-valued  
$({\mathscr F}_t)$-$Q$-Wiener process with $\W^Q_0=0$, $\mathbb{P}$ a.s. 
In this case  the It\^o-type  integral  with respect to $\W$  extends to a larger class,
see Chapter  4.2 and 4.3 of \cite{DaPratoZabczyk92}. 
If $\Y$ is a  predictable process with values in $ \mathcal{L}_2(U_0; H)$
with some integrability properties, 
then  the It\^o-type  integral of $\Y$ with respect to $\W$,
 i.e. $\int_0^t \Y_r d \W_r, t \in [0,T],$ is well-defined.

\begin{prop}
\label{PChainRule-2}
Let  $\M$ be a  process  of the  form
\begin{equation} \label{AA}
\M _t = \int_0^t  \Y_r d \W^Q_r, t \in [0,T],
\end{equation}
where $\Y$ is an $\mathcal{L}(U;H)\cap \mathcal{L}_2(U_0; H)$-valued predictable process
such that 
\begin{equation} \label{AAAbis}
\int_0^T {\mathrm Tr} [\Y_r Q^{1/2} (\Y_r Q^{1/2})^*]  d r  < 
\infty \ {\mathrm a.s.}
\end{equation}
Then $\M$ is a $H$-valued local martingale. Moreover we have the following.
\begin{itemize}
 \item[(i)] If $\X$ is an $H$-valued predictable process such that 
\begin{equation} \label{EChainRule11}
  \int_0^T \langle \X_r, \Y_r Q^{1/2} (\Y_r  Q^{1/2})^* \X_r \rangle_H d r 
< \infty, \,\, { a.s.},
\end{equation}
then, using Riesz identification, 
\begin{equation} \label{AA1}
 N_t = \int_0^t  \X^\ast_r  d \M_r, t \in [0,T],
\end{equation}
is a real local martingale. If the expectation of \eqref{EChainRule11} is finite, then  $N$ is a square integrable martingale.
\item[(ii)] If, for some separable Hilbert space $E$, $\K$ is a
 $\mathcal{L}(H,E)$-valued, $(\mathscr{F}_t)$-predictable process such that 
\begin{equation} \label{AAAter}
\int_0^T {\mathrm Tr} [\K_r\Y_r Q^{1/2} (\K_r\Y_r Q^{1/2})^*]  d r  < \infty
\ {\mathrm a.s.},
\end{equation}
 then
the $E$-valued It\^o-type stochastic integral $\int_0^t \K d \M, 
t \in  [0,T]$,
 is well-defined, it is a local martingale and it equals $\int_0^t \K \Y d \W^Q$.
\end{itemize}
\end{prop}
\begin{proof}
The results above are a consequence of \cite{DaPratoZabczyk92} Section 4.7.
at least when the expectations of \eqref{AAAbis},
\eqref{EChainRule11} and \eqref{AAAter}
are finite.
In particular the first part is stated in Theorem 4.12 of 
 \cite{DaPratoZabczyk92}.
Otherwise, on proceeds by localization, via stopping arguments.
\end{proof}

\begin{rem} \label{R24}
  In the sequel we will also denote the integral in \eqref{AA1}, by
$\int_0^t \langle \X_r, d \M_r \rangle_H, t \in [0,T],$
or by $\int_0^t {}_{H^*}\langle \X^\ast_r, d \M_r \rangle_H, t \in [0,T],$
using again Riesz identification.
\end{rem}

\section{Finite dimensional calculus via regularization}
\label{finreg}

\subsection{Integrals and covariations}

This theory has been developed in several papers, starting from \cite{rv, rv1}.
A survey on this subject is given in \cite{Rus05}.
The formulation is light, efficient when the integrator 
is a finite quadratic variation process, but 
it extends to many integrator processes whose paths have a $p$-variation
with $p > 2$.
  Integrands are allowed to be 
  anticipating and the integration theory and calculus appears to be 
close to a pure 
pathwise approach even though there is still a probability
space behind. The  theory  clearly allows  non-semimartingales
integrators.\\
Let now $X$ (resp. $Y$) be a real continuous (resp. a.s. integrable) process,
both indexed by $t \in [0,T]$. 
\begin{dfn}
Suppose that, for every $t\in[0,T]$, 
 the following limit 
\begin{equation} \label{DPFI}
\int_{0}^{t}Y_{r}d^{-}X_{r}:=\lim_{\epsilon\rightarrow 0}
\int_{0}^{t} Y_{r}\frac{X_{r+\epsilon}-X_{r}}{\epsilon}dr,
\end{equation}
exists in probability. 
If the obtained  random function
admits a continuous modification, that process is 
 denoted by  $\int_0^\cdot Yd^-X$ and called  
\textbf{(proper)  forward integral of $Y$ with respect to $X$}.
\end{dfn}
\begin{dfn} \label{D32}
If the limit \eqref{DPFI} exists in probability for every $t \in [0,T[$
and $\lim_{t\rightarrow T}\int_{0}^{t}Yd^{-}X$ exists in probability, 
the limiting 
random variable is called the \textbf{improper forward integral of $Y$ 
with respect to 
$X$} and it is still denoted by
$\int_{0}^{T} Y d^{-}X$.
\end{dfn}
When $ p > 2$, in general $\int_0^\cdot X d^- X$, does not exist, in particular
when $Y = X$. In that case, in stochastic calculus naturally appears the
{\it symmetric} (generalized Stratonovich) integral, see for instance \cite{Rus05}.

As we mentioned, the covariation is a crucial notion in stochastic calculus via
regularization.

\medskip
\begin{dfn}\label{D33}
The \textbf{covariation of $X$ and $Y$} is defined by 
\begin{equation*}							
\left[X,Y\right]_{t} = \left[Y,X\right]_{t}  = \lim_{\epsilon\rightarrow 0^{+}} 
\frac{1}{\epsilon} \int_{0}^{t} (X_{s+\epsilon}-X_{s})(Y_{s+\epsilon}-Y_{s})ds, t \in [0,T],
\end{equation*}
if the limit exists in probability 
% in the ucp sense with respect to
for every  $t \in [0,T]$, provided that 
the limiting random function  admits a continuous version.
If $X=Y,$ $X$ is said to be \textbf{finite
quadratic variation process} and we set $[X]:=[X,X]$.
A vector $(X^1, \ldots, X^n)$ of real processes is said to admit
{\bf all its mutual brackets} if $[X^i, X^j], 1 \le i, j \le n,$
exist. 
\end{dfn}
One natural question arises.
 What is  the link between
the regularization and discretization techniques of F\"ollmer (\cite{fo}) 
type?
Let $Y$ be a cadlag process. One alternative method could be to define
$\int_0^T Y dX $ as the limit 
of 
$$ \sum_{i=0}^{N-1} Y_{t_i} (X_{t_{i+1}} - X_{t_{i}}), N \in \N^\ast, $$
when the mesh $\max_{i=0}^{N-1}  (t_{i+1} - t_{i})$    of the subdivision
\begin{equation} \label{subdivision}
 0 = t_{0} < \ldots < t_N = T,
\end{equation}
converges to zero.
A large part of calculus via regularization can be essentially translated in that formal  language via discretization. 
However, even if it is not essential, we decided to 
keep  going on with regularization methods.
First, because
 that approach is direct and analytically efficient.
 Second, in many contexts, the class of integrands
 is larger. Let us just  fix one simple example: the Wiener integral with respect to
Brownian motion. Let $g \in L^2([0,T])$ and $W$ be a classical Wiener process; 
$\int_0^t g d^-W, t \in [0,T]$ exists and equals Wiener-It\^o integral
$\int_0^t g dW, t \in [0,T]$.
However the  discretizations limit of
 $ \sum_{i=0}^{n-1} g({t_i}) (W_{t_{i+1}} - W_{t_{i}}) $
 may either not exist, or depend on the sequences of subdivisions.
 Indeed, as an example, let us choose $g = 1_{\Q \cap [0,T]}$, where $\Q$ is the
set of rational numbers. 
If,  all the $t_i$ elements of   subdivision \eqref{subdivision}
  where irrational (except for the extremities), then  the limit would be zero,
as for the It\^o-Wiener integral, being $g = 0$ a.e.
If on the contrary, all of the $t_i$ are rational, then 
the limit is $W_T -W_0$. 

\medskip 

In the proposition below we list some  properties
relating It\^o calculus and forward calculus, see e.g.
\cite{Rus05}.

\begin{prop} \label{Pproperties}
 Suppose that $M$ is a  continuous $(\mathscr{F}_{t})$-local martingale and $Y$ is 
cadlag and predictable. Let   $V$ be a  bounded variation process.
Let $S^{1}$, $S^{2}$ be $(\mathscr{F}_{t})$-semimartingales with 
decomposition $S^{i}=M^{i}+V^{i}$, $i=1,2$,
 where $M^{i}, i = 1,2$ are  $(\mathscr{F}_{t})$-continuous local martingales and
 $V^{i}$ continuous adapted bounded variation processes. We have the following.
\begin{enumerate}
\item  $M$ is a finite quadratic variation process and
$[M]$ is the classical bracket $\langle M\rangle$. 
\item $\int_0^\cdot Yd^- M$ exists and it equals  
It\^o integral $\int_0^\cdot YdM$.
\item  Let us suppose $V$ to be continuous 
and $Y$ cadlag (or vice-versa); then
 $[V] =[Y,V]= 0$. Moreover $\int_0^\cdot Y d^-V=\int_0^\cdot Y dV $,  
is the {\bf Lebesgue-Stieltjes integral}.
\item $[S^{i}]$ is the classical bracket and $[S^{i}]=\langle M^{i}\rangle $.
\item $[S^{1},S^{2}]$ is the classical bracket and $[S^{1},S^{2}]=\langle
 M^{1},M^{2}\rangle $. 
\item If $S$ is a continuous semimartingale and $Y$ is cadlag and adapted, then
$
\int_0^\cdot Yd^- S=\int_0^\cdot YdS 
$
is again an It\^o integral.
\item If $W$ is a Brownian motion and $Y$ is progressively measurable process 
such that $\int_0^T Y^2_s ds < \infty$ a.s., then 
$\int_0^\cdot Yd^- W$ exists and equals the It\^o integral  $\int_0^\cdot YdW$.
\end{enumerate}
\end{prop}

Coming  back to the general calculus we state 
the {\bf integration by parts} formula, see e.g. item 4) of Proposition 1 in \cite{Rus05}.
\begin{prop} \label{IBP}
Let $X$ and $Y$ be continuous processes. Then 
\begin{equation*}		
Y_{t}X_{t}=Y_{0}X_{0}+\int_{0}^{t}Yd^{-}X+\int_{0}^{t}Xd^{-}Y+[X,Y]_{t},
\end{equation*}
provided that two of the three previous integrals or covariation
exist. If $X$ is a continuous bounded variation process, then
$ \int_{0}^{t}Xd^{-}Y=Y_{t}X_{t}-Y_{0}X_{0}-\int_{0}^{t} Y dX$.
\end{prop}

The kernel of calculus via regularization is
It\^o formula. It is a  well-known result in the semimartingales theory,
but it also  extends to the framework of  finite quadratic variation processes.
Here we only remind the one-dimensional case, in the form of a It\^o 
chain rule. It is essentially a consequence of
Proposition 4.3 of \cite{rv4}.

\begin{thm}		\label{ITOFQV}
Let $F:[0,T]\times \R\longrightarrow \R$ such that
 $F\in C^{1,2}\left( [0,T[\times \R \right)$ and $X$ be 
a finite quadratic variation process. 
We set $Y_t = F(t, X_t), t \in [0,T]$.
Let $Z = (Z_t, t \in [0,T])$ be an a.s. bounded    process.
We have
\begin{equation} \label{ItoChain}
\int_{0}^{t}  Z_r d^- Y_r      = 
 \int_{0}^{t}
Z_r \partial_{r} F(r,X_{r})dr + \int_{0}^{t} Z_r \partial_{x} F(r,X_{r})d^-X_r +
\frac{1}{2}\int_{0}^{t} Z_r \partial^2_{x\, x} F(r,X_{r})d[X]_{r},
\end{equation}
in the following sense: if the first (resp. the third) integral 
exists then the third (resp. the first) exists and
formula \eqref{ItoChain} holds. 
\end{thm}
Taking $Z = 1$, comes out the natural  It\^o formula below.
\begin{prop} \label{PITO} With the same assumptions of Theorem \ref{ITOFQV}
we have  
$$  \int_{0}^{t} \partial_{x} F(r,X_{r})d^-X_r = 
 F(t,X_{t}) - F(0,X_{0}) - 
 \int_{0}^{t}\partial_{r}  F(r,X_{r})dr 
- \frac{1}{2}\int_{0}^{t} \partial^2_{x\, x} F(r,X_{r})d[X]_{r}.$$
\end{prop}
 Theorem \ref{TITOF} will extend the formula above to the case of
Banach space valued integrators.

\medskip
An  adaptation of Proposition 11 of \cite{Rus05} and Proposition 2.2
of \cite{rg1}   gives the following.
Given a real interval $I$ and $h: I \rightarrow \R$ be a bounded variation
function, we denote by $\Vert h \Vert_{\mathrm var}$ the total variation of 
$h$. 
\begin{prop} \label{BVM} 
Let $I$ be a real interval and $f, g:[0,T] \times I  \rightarrow \R$
of class $C^{0,1}([0,T] \times I)$. 
Moreover for $ h =f$ or $h = g$ we suppose the following.
\begin{itemize}
\item For every $x \in I$, $h(\cdot, x)$ has bounded variation.
\item For any compact subset $K$ of $I$, there is a non-negative measure $\nu_K$
on ${\mathcal B}([0,T])$ such that for every Borel subset $B$ of $[0,T]$
we have
$$  \sup_{x \in K} \left \vert \partial h(B, x) \right \vert \le \nu_K(B).$$
\end{itemize}
 Let
 $X$ and $Y$ be two real processes such
that $(X,Y)$ admits all its mutual brackets.
Then $[f(\cdot, X), g(\cdot, Y)]_t = \int_0^t \partial_x f(s, X_s) 
\partial_x g(s, Y_s) d[X,Y]_s.$
\end{prop}
\begin{rem} \label{RExample}
\begin{itemize}
\item A typical example of a function $f$ or $g$ given in Proposition \ref{BVM} is
$h(r,x) = \sum_{i=1}^n \ell_i(r) g_i(x)$, where $g_i$ are of 
class $C^1(I)$ 
and all the $\ell_i: [0,T] \rightarrow \R$ have bounded variation.
\item Another possibility is $h$ of class $C^1([0,T] \times I)$. 
\end{itemize}
\end{rem}

Below we introduce the notion of weak Dirichlet process
 which was introduced in \cite{er2} and \cite{rg2}.
\begin{dfn}
\label{def:Dirweak}
A real continuous process $X \colon [0,T]\times \Omega \to \mathbb{R}$ is called {\bf weak Dirichlet process} 
if it can be written as
\begin{equation}
\label{eq:dec2}
X=M+A,
\end{equation}
where
\begin{itemize}
\item[(i)]  $M$ is a continuous local martingale,
\item[(ii)]  $A$ is a process such that $\left[ A,N\right] =0$ for every 
continuous local martingale $N$ and $A_0=0$.
\end{itemize}
\end{dfn}

\begin{prop}
\label{rm:ex418}
\begin{enumerate}
\item
The decomposition described in Definition \ref{def:Dirweak} is unique.
\item A real continuous semimartingale $S$ is a weak Dirichlet process.
\end{enumerate}
\end{prop}
\begin{proof}
1. is stated in Remark 3.5 of \cite{rg2}.
2. is obvious since 
 a bounded variation process $V$ is a zero quadratic variation process
by item 3. of Proposition \ref{Pproperties}.
\end{proof}

\subsection{The deterministic  calculus via regularization.}

\label{sdetint}

An useful particular case arises  when $\Omega$ is a singleton, i.e.
when the calculus becomes deterministic. 
\\
 We will essentially concentrate in 
the definite integral  on an interval 
 $J = ]a,b]$,
where  $ a < b$ are two real numbers.
 Typically, in our applications we will consider
$a = -\tau$ or $a = - t$ and $b= 0$. 
 That integral will be a real number, instead of functions.

We start with a convention. If $f: [a,b] \rightarrow \R$ is a cadlag function,
we extend it naturally to another cadlag function $f_J$ on real line setting 
\begin{equation*} 
f_J (x) = \left \{
\begin{array}{ccc} 
f(b) &:& x > b  \\
 f(x)  &:& x  \in [a,b]   \\
 0  &:& x  < a.   
\end{array}
\right.
\end{equation*}
If $g$ is finite Borel measure on $[0,T]$,
we define the {\bf deterministic forward integral}
   $\int_{]a,b]} g(dx) d^-f(x)$ (or simply $\int_{]a,b]} g d^-f$)
  as the limit of 
$ \int_{]a,b]}  \frac{g(ds)}{\varepsilon}
 (f_J(s+\varepsilon) - f_J(s)) $, 
when $\varepsilon \downarrow 0$, provided it exists.
In most of the cases $g$ will be absolutely continuous
whose density will be still denoted by the same letter.
A similar definition can be provided for the (deterministic) covariation 
of $[f,g]$ of two (continuous) functions
$f$ and $g$ defined on some  interval $I$.
 Without restriction of generality, we suppose that $0 \in I$.
We set $[g,f](x), \ x \in I$,
the pointwise limit (if it exists), when $\varepsilon \rightarrow 0$ of 
$$ \int_0^x (g(r+\varepsilon) - g(r))  (f(r+\varepsilon) - f(r)) \frac{dr}{\varepsilon}, x \in  I.$$
If $g = f$, we also denote it by $[f]$.

\begin{rem}\label{BVI} The   
 following statements follow directly from the definition and
are left to the 
reader. The reader may consult \cite{rv93} for similar considerations.
By default, the bounded variation functions will be considered as cadlag. 
\begin{enumerate}
\item If $f$ has bounded variation
then $\int_{]a,b]}  g(s) d^-f(s)$ 
 is the classical Lebesgue-Stieltjes integral  $\int_{]a,b]} g df $. 
In particular, if $g = 1$,  $\int_a^b g(s) d^-f(s) = f(b) - f(a). $
\item If $g$ has bounded variation, the following integration by parts
 formula holds:
 $\int_{]a,b]}  g(s) d^-f(s)$ 
equals 
 $g(b-) f(b)  - \int_{]a,b]} f(s) dg(s)$.
\item A deterministic version of Theorem \ref{ITOFQV} can be easily stated,
with respect to integrals of the type  $\int_{]a,b]}$
instead of $\int_0^t$.
\end{enumerate}
\end{rem}
Besides $B= C([-T,0])$,
we introduce another Banach space.
Given a continuous function $g:[-T,0] \rightarrow \R$ we define the
 $2$-{\bf regularization variation} 
by $\vert g \vert_{2,var} := \sup_{0 < \varepsilon < 1} \int_{-T}^0 \left 
(g(s+\varepsilon) - g(s) \right )^2 \frac{ds}{\varepsilon}$.
We define by $V_2$ the space of $g \in B$ such that $\vert g \vert_{2,var}$ is
 finite.
If $\eta \in C([-T,0])$, we denote $\vert \eta \vert_\infty := \sup_{x \in [-T,0]}
 \vert \eta (x)\vert$.
\begin{prop} \label{V2Alter}
The functional $g \mapsto  \vert g \vert_\infty   + \vert g \vert_{2,var}$
 is a norm on $V_2$. 
Moreover $V_2$, equipped with that norm, is a Banach space.
\end{prop} 
\begin{proof} \ To prove that  $\vert \cdot \vert_{2,var}$ 
is a norm, the only non-obvious property is
 the triangle inequality.
That follows because of the triangle inequality related to the 
$L^2([-T,0])$-norm.
It remains to show  that any Cauchy sequence in $V_2$ converges to an element of $V_2$.
Let $(g_n)$ be such a sequence. Since $C([-T,0])$ is a Banach space, there is $g \in C([-T,0])$
such that $g_n$ converges uniformly to $g$. 
Let $M >0$.
Since $(g_n)$ is a Cauchy sequence with respect to $\vert \cdot \vert_{2,var}$,
 there is $N$ such that if $n, m \ge N$, with 
$$  \int_{-T}^0 \left ((g_n - g_m)(r+\varepsilon) - (g_n - g_m)(r) \right )^2 \frac{dr}{\varepsilon} \le M,$$
for every $0 < \varepsilon < T$.
Let us fix  $0 < \varepsilon < T$. 
Choosing $m = N$ in previous expression   and letting $n$ go to $\infty$ 
it follows that 
\begin{eqnarray*} 
 \int_{-T}^0 \left (g(r+\varepsilon) - g(r) \right )^2 \frac{dr}{\varepsilon} &\le&
2 \int_{-T}^0 \left (
(g - g_N)(r+\varepsilon) - (g - g_N(r)) \right )^2 \frac{dr}{\varepsilon}   \\
&+& 2 \int_{-T}^0 \left (g_N(r+\varepsilon) - g_N(r)\right )^2 \frac{dr}{\varepsilon}   \\
&\le & 2 M +  2 \int_{-T}^0 \left (g_N(r+\varepsilon) - g_N(r) \right )^2 \frac{dr}{\varepsilon}.
\end{eqnarray*}
Taking the supremum on $0 < \varepsilon < T$, we get that $\vert g \vert_{2,var}$ is finite
and the result follows.
\end{proof}
$V_2$ is a Banach subspace of $B$. 
Given a continuous function $\psi: [0,T] \rightarrow \R$ be a continuous 
increasing function such that $\psi(0) = 0$,
we  denote by $V_{2,\psi}$ the space of functions $\eta:[-T,0] \rightarrow \R$ 
such that $[\eta]$ exists and equals  $ \psi$.
\begin{prop}\label{P313} 
$V_{2,\psi}$ is a closed subspace of $V_{2}$. 
\end{prop}
\begin{proof} \
  Let $(g_n)$ be a sequence in $V_{2,0}$ i.e. such that  $[g_n](x), x \in [-T,0]$ exists and
 equals   $\psi$.
We suppose that $g_n$ converges to $g$ in $V_2$. 
Now, for fixed $\varepsilon > 0$, $x \in [-T,0]$ we consider
\begin{equation}\label{IntCov}
I_\psi(\varepsilon,x) := -\int_{x}^0 dr  \left (g(r+\varepsilon) - g(r) \right )^2 \frac{dr}{\varepsilon} - \psi(x).
\end{equation}
We want to prove that for every  $x \in \R$,   $I_\psi(\varepsilon,x)$
converges to $0$,
when $\varepsilon \rightarrow 0+$. 
The left-hand side of \eqref{IntCov}  is bounded by
 $4 I_1(\varepsilon,N,x) + 4 I_2(\varepsilon,N,x) +4 I_3(N,x)$,
where, for $x \in [-T,0], N \in \N^\ast$,
\begin{eqnarray} \label{IN}
I_1(\varepsilon,N,x) &=&  \left \vert \int_{x}^0 \left (
(g - g_N)(r+\varepsilon) - (g - g_N(r) \right )^2 \frac{dr}{\varepsilon} 
\right \vert  \nonumber \\
I_2(\varepsilon,N,x) & =& \left \vert \int_{x}^0 \left(g_N (r+\varepsilon) - g_N(r)\right)^2 
\frac{dr}{\varepsilon} 
- [g_N](x)  \right \vert  \\
I_3 (N,x)&=& \vert [g_N](x) - \psi(x)\vert. \nonumber 
\end{eqnarray}
 We fix $x \in [-T,0]$. 
 Since  $g_N \in V_{2,\psi}$ then $[g_N] = [g] = \psi$ and $I_3(N,x)$ equals zero.
Let $\delta > 0$ and $N$
such that for every $0 < \varepsilon <T$, 
$I_1(\varepsilon, N,x) \le \delta $.
Choose $\varepsilon_0$ such that 
$ I_2(\varepsilon, N,x) \le \delta $
if $0 < \varepsilon < \varepsilon_0$.
Consequently for  $0 < \varepsilon < \varepsilon_0$,
then we have $I_\psi(\varepsilon,x) \le 2\delta $.
This shows that  $I_\psi(\varepsilon,x)$ converges to zero and
so  $V_{2,\psi}$ is a closed subspace of $V_2$.
\end{proof}

\section{About infinite dimensional classical stochastic calculus}

\subsection{Generalities}
\label{IDCSCG}

Infinite dimensional stochastic calculus is an important tool for studying
  properties related to stochastic evolution problems,
 as stochastic partial differential equations, stochastic functional equations, as delay equations.
When the evolution space is Hilbert a lot of work was performed, see typically
 the celebrated monograph of G. Da Prato and J. Zabczyk
 \cite{DaPratoZabczyk92}, 
in particular Section \ref{DBZI} mentions the corresponding notion
of stochastic integral.
An alternative, similar approach, is the one related to
 random  fields, see e.g.  \cite{Walsh86} and \cite{Dalang99}. 
Infinite dimensional stochastic calculus has been also developed
in the framework of  Gelfand triples, used for instance in
 \cite{pardouxthesis}.
Contributions exist also for  Banach space valued 
stochastic integrals, see  \cite{brez, det, det1, vanNervUMD}, 
where  the situation is more involved
than in the Hilbert framework:
the so-called reproducing kernel space cannot   be 
described as $Im (Q^{1/2})$, as in Section \ref{DBZI}, and the notion of
 Hilbert-Schmidt operator has to be substituted 
with the one  of $\gamma$-radonifying. \\
The aim of our approach is to try to introduce suitable techniques which
allow to treat typical infinite dimensional processes similarly to finite-dimensional
 diffusions.
As we mentioned, stochastic process with values in infinite dimensional
spaces will be indicated by a bold  letter of the type $\X, \Y, \Z$  and so on.
Let $B$ be a separable Banach space and  $\mathbb{X}$ be a $B$-valued process.
Consider $F:B\longrightarrow \mathbb{R}$ be of class $C^{2}$ in the  Fr\'echet sense.
One may ask what could be a good It\^o formula in this framework.
 We are interested in an It\^o type expansion  of $F(\mathbb{X})$,
keeping in mind that, classically, It\^o formulae contain an integral term
 involving second order type derivatives 
and a quadratic variation.
We first introduce some classical notions of quadratic variation
very close to those of the literature, see \cite{dincuvisi, MetivierPellaumail80, Metivier82}, but in the spirit of calculus via regularization.
 Those above mentioned authors  introduce in fact two quadratic variations:
the {\it real} and the {\it tensor} quadratic variation.
The definition below is a reformulation in terms of 
regularization of the {\it real} quadratic variation of $\X$.
We prefer here, to avoid possible confusions, to replace
the denomination {\it real} with {\it scalar}.
\begin{dfn}
\label{def:scalarquadraticvariation}
Consider a separable Banach space $B$.
We say that a (strongly) measurable process $\X\colon [0,T]\times \Omega \to B$
 admits a {\bf scalar quadratic variation} if, for any $t\in [0,T]$,
 the limit, for $\epsilon\searrow 0$ of 
\[
[\X,\X]^{\epsilon, \mathbb{R}}_t:= \int_{0}^{t} \frac{\left 
|\X_{r+\epsilon}-\X_{r} \right |^2_B}{\epsilon} dr,
\]
exists in probability and it admits a continuous version. The limit process 
is called {\bf scalar quadratic variation} of $\X$ 
and it is denoted by $[\X,\X]^{\mathbb{R}}$. 
\end{dfn}

In Definition 1.4 of \cite{DGR1}  the authors introduce the following definition.
\begin{dfn}\label{def:tensorcovariation}
Consider two separable Banach spaces $B_1$ and $B_2$.
Suppose that either $B_1$ or $B_2$ is different from $\mathbb{R}$.
Let  $\X\colon [0,T]\times \Omega \to B_1$ and 
$\Y\colon [0,T]\times \Omega \to B_2$  be two (strongly) 
measurable processes.
We say that $(\X,\Y)$ admits a {\bf tensor covariation} if the limit,
 for $\epsilon\searrow 0$ of the $B_1\hat\otimes_\pi B_2$-valued processes
\[
[\X,\Y]^{\otimes, \epsilon}:= \int_{0}^{\cdot} \frac{ \left (\X_{r+\epsilon}-\X_{r} \right ) \otimes
\left (\Y_{r+\epsilon}-\Y_{r}\right )}{\epsilon}dr
\]
exists in the ucp sense (i.e. uniform convergence in probability).
 The limit process is called {\bf tensor covariation} of $(\X,\Y)$ 
and is denoted by $[\X,\Y]^\otimes$. The tensor covariation $[\X,\X]^\otimes$ is called 
{\bf tensor quadratic variation} of $\X$ and denoted by $[\X]^\otimes$.
\end{dfn}

\begin{rem} \label{R7.7} Let $\X,  \Y$  be 
measurable processes 
defined on $[0,T] \times \Omega  $ 
with values respectively on $B_1$ and $B_2$.
We  have the
following. 
\begin{enumerate}
\item If $\X$  has a zero scalar quadratic variation 
and $\Y$ has a scalar quadratic variation then 
$[\X, \Y]^\otimes =  [\Y, \X]^\otimes = 0$.
Moreover $\X + \Y$ has a scalar quadratic variation and $[\X + \Y ]^\R = [\Y]^\R $;
\item If $\X$ is a bounded variation process 
 then $\X$ admits a zero 
scalar quadratic variation. 
 \item 
 Let $\M$ be a local martingale with values in a separable Hilbert space $H$.
Then it has a scalar quadratic  variation.
\item  If $B=\R^{n}$ the space $B \hat \otimes_\pi B$ is associated with
 the space of $n \times n$
real matrices, as follows. Let $(e_i, 1 \le i \le n)$ be the canonical
 orthonormal basis of 
$\R^n$. A matrix $A = (a_{ij})$ is naturally associated 
with the element $\sum_{i,j=1}^n a_{ij} e_i \otimes e_j$. 
 $\mathbb{X}=(X^{1},\ldots,X^{n})$ admits all its mutual covariations
if and only if $\mathbb{X}$ admits a tensor quadratic variation.
Moreover  $[\X, \X] = \sum_{i,j=1}^n [X_i, X_j] e_i \otimes e_j$. 
\end{enumerate}
Items 1. and 2. are easy to establish. Item 3. is stated in Remark 4.9 of 
\cite{RusFab}. Item 4. constitutes an easy exercise, but it was stated in 
Section 6.2.1 of \cite{DGR}.
\end{rem}

Let us consider now 
$F:B\longrightarrow \mathbb{R}$ be of class $C^{2}$.
 In particular  $DF: B\longrightarrow  \shl(B;\R):=B^{\ast}$ 
and  $D^{2}F:B \longrightarrow \shl(B; B^{\ast})\cong 
{\mathcal Bi}(B,B)\cong (B\hat{\otimes}_{\pi}B)^{\ast}$
are continuous. 
As a first attempt, we expect to obtain  an It\^o formula type expansion
 of the following type.
\begin{equation}
\label{FITOAtt} 
F(\mathbb{X}_{t})
= F(\mathbb{X}_{0})+ ''\int_{0}^{t} 
{}_{{B^{\ast}}}{\langle} DF(\mathbb{X}_{s}), d\mathbb{X}_{s}\rangle_{B} \; ''  
+ \frac{1}{2}  \int_{0}^{t} 
{}_{(B\hat{\otimes}_{\pi}B)^{\ast}}{\langle} D^{2}F(\mathbb{X}_{s}),
 d[\mathbb{X}]_{s}\rangle_{B\hat{\otimes}_{\pi}B}.
\end{equation}
This supposes of course that the tensor covariation $[\X,\X]^\otimes$ 
exists and it has bounded variation. A reasonable sufficient condition for this
demands that the scalar quadratic variation $[\X, \X]^\R$ exists. 
A formal proof of the It\^o formula, inspired from the one-dimensional
case could be the  following. Let $\varepsilon > 0$. 
We have
\[
\int_{0}^{t} \frac{F(\mathbb{X}_{s+\epsilon})-F( \mathbb{X}_{s}) }
{\epsilon} ds \xrightarrow[\epsilon\rightarrow 0] 
 {ucp}F(\mathbb{X}_{t})-F( \mathbb{X}_{0}), t \in [0,T]. 
\]
By a Taylor's expansion, the left-hand side equals the sum  
\[
 \int_{0}^{t}
{}_{B^{\ast}}{\langle}DF(\mathbb{X}_{s}), \frac{\mathbb{X}_{s+\epsilon}- \mathbb{X}_{s} }{\epsilon}\rangle_{B} ds \; + 
\int_{0}^{t}
%\prescript{}
{}_{(B\hat{\otimes}_{\pi}B)^{\ast}}{\langle}D^{2}F( \mathbb{X}_{s}),  \frac{(\mathbb{X}_{s+\epsilon}- \mathbb{X}_{s})\otimes^{2} }
{\epsilon}\rangle_{B\hat{\otimes}_{\pi}B} ds + R(\epsilon,t),
 \]
where $R(\varepsilon, \cdot)$ converges ucp to zero.
Consequently, previous formal proof requires
 a good  notion
of quadratic variation.
Moreover the first (stochastic) integral needs to be defined.
The following natural obstacles    appear. 
\begin{itemize}
\item In many interesting cases mentioned at the beginning
of Section \ref{IDCSCG}, $\X$ 
 is not a semimartingale, and it
 has not even a scalar and tensor quadratic variations. 
\item Stochastic integration, when the integrator takes values 
in a Banach space
is not an easy task.
\end{itemize}

\bigskip

\subsection{Tensor covariation and operator-valued covariation}
\label{TensDPZ}

In Definition \ref{def:tensorcovariation} we introduced the notion of
 tensor covariation in the spirit of M\'etivier and Pellaumail.
Before proceeding  and introducing the more general definition of 
\emph{$\chi$-covariation} we devote some space recalling another
 (somehow classical) 
definition used for example by several authors in stochastic calculus 
in Hilbert spaces, as
  Da Prato and Zabczyk.

Let $H$ be a separable Hilbert spaces
and $\M,\N$ be two $H$-valued continuous
 local martingales.% $\M$ [resp. $\N$].
The first (tensor) covariation was denoted by $[\M,\N]^\otimes$,
the second one will be denoted by $[\M,\N]^{cl}$.
A first difference arises by the fact 
 $[\M,\N]^\otimes$ takes values in $H \hat\otimes_\pi  H$ 
and  $[\M,\N]^{cl}$ lives in $\shl_1(H)$.

We remind from Section \ref{SPrelim1} that every  element 
$u \in H\hat\otimes_\pi H$  is isometrically associated 
with an element $T_u$ in the space of nuclear operators  
$\shl_1(H)$,
 so it makes sense to compare  Definition \ref{def:tensorcovariation} and the definition of
 operator-valued covariation. 

\begin{dfn} \label{DZCovariation}
Let $\X$ and $\Y$ be two $H$-valued continuous processes. 
We say that $(\X, \Y)$ admits an {\bf operator-valued
 covariation}, denoted by $[\X,\Y]^{cl}$, if there exists
a bounded variation process $\V$ with values in $\shl_1(H)$, 
denoted by $[\X,\Y]^{cl}$,   such that, 
for every $a, b  \in H $,
the covariation (in the sense of regularization) of 
$\langle a, \X \rangle$ and $\langle b, \Y \rangle$  equals $\langle \V a, b \rangle$.

In the sequel we will also use the notation $[\X, \Y]^{cl}(a, b) :=  \langle \V a, b \rangle$. In other words 
the continuous linear functional $[\X, \Y]^{cl}(a, \cdot)$ is Riesz-identified to $\V a$.
We will of course identify without further mention $\V(a)$ and $[\X,\Y]^{cl}(a, \cdot)$. 
\end{dfn}
\begin{rem} \label{RCovDZ}
\begin{enumerate}
\item If $\X$ and $\Y$  are local martingales then 
$\langle a, \X \rangle$ and $\langle b, \Y \rangle$ are real local martingales
and previous covariations in the sense of regularization
are classical covariations of martingales, see Proposition 2.4 (3) 
of \cite{rv93} and item 5. of Proposition \ref{Pproperties}.
\item If $\X = \Y$  is a local martingale and $\V$ is $[\X, \X]^{cl}$,
 then, by Doob-Meyer decomposition,
 $\V$ fulfills the following property.
For every $a, b \in H$,  we have
  $\langle a, \X \rangle   \langle b, \X \rangle - 
 \langle \V a, b \rangle$ is a local martingale and obviously
 $\langle \V_\cdot (a), a \rangle  \ge 0 $ is a non-negative increasing
process for every $a \in H$;
in particular, for every $t \in [0,T]$, $\V_t$ is a non-negative 
map in $\shl_1(H)$.
\item   Proposition 3.12 in \cite{DaPratoZabczyk92} states   that for
 a continuous square integrable martingale $\X$, 
the quadratic variation exists (and is unique). By stopping arguments, this can be extended to every local martingale $\X$.
\end{enumerate}
\end{rem}
The proposition below illustrates some relations between the tensor covariation and the operator-valued covariation.
\begin{prop} \label{PDZTens}
\begin{enumerate}
\item The operator-valued covariation is unique.
\item If $(\X,\Y)$  admits a tensor covariation then, it also has
 an  operator-valued covariation
and, after the identification above between $H \hat \otimes_\pi H$ 
 and $\shl_1(H)$, they are equal.\\
 In particular, for every $a \in H, b \in H$,
$\langle j(a^*\otimes b^*), [\X, \Y]^{\otimes}\rangle
= [\langle \X,a \rangle, \langle \Y, b \rangle]$.
 \item If $\X$ and $\Y$ are local martingales then they admit  a scalar
quadratic variation. Moreover, $(\X, \Y)$
admits tensor and operator-valued covariations.
\end{enumerate}
\end{prop}
\begin{proof}
Let $\varepsilon >0$. Taking into account Lemma \ref{lm:aggiunto}, choosing $\varphi \in (H \hat \otimes_\pi H)^*$ of the type $\varphi = j (a^* \otimes b^*)$ where $a,b,\in H$,  we have
\begin{multline}
\label{DZ1}
{\phantom{\bigg (}}_{(H\hat\otimes_\pi H)^*} 
\left\langle \varphi,
 \frac{1}{\varepsilon}\int_0^t (\X_{s+\varepsilon} - \X_s) \otimes 
( \Y_{s+\varepsilon} -  \Y_s) ds \right \rangle_{H\hat\otimes_\pi H} \\
= 
\frac{1}{\varepsilon} 
\int_0^t ds  {\phantom{(}}_{(H\hat\otimes_\pi H)^*} 
\left\langle \varphi,
 (\X_{s+\varepsilon} - \X_s) \otimes ( \Y_{s+\varepsilon} -  \Y_s)
 \right \rangle_{H\hat\otimes_\pi H} 
 = 
 \frac{1}{\varepsilon} 
\int_0^t ds  \langle \X_{s+\varepsilon} - \X_s , a\rangle \langle  \Y_{s+\varepsilon} -  \Y_s, b \rangle.
\end{multline}
So the first expression of the equality above converges if  
the covariation of the real processes 
$\langle \X , a\rangle$ and  $\langle  \Y, b \rangle$ exists.
\begin{enumerate}
\item Let be two $\shl_1(H)$-
valued processes $\V^1, \V^2$ verifying
$$ \langle \V^i(a), b \rangle =   [\langle \X, a \rangle,\langle \Y, b
 \rangle ],  i = 1,2, $$
for every $a, b \in H$.
Let $\U^i$ be the associated process with values in  $H\hat\otimes_\pi H$
in the sense of the  usual isomorphism \eqref{Isomorphism}
   between  $H\hat\otimes_\pi H$ and $\shl^1(H)$.
Then, taking into account \eqref{DZ1}, for every $t \in [0,T]$ we have
$\langle \varphi, \U^1_t \rangle = \langle \varphi, \U^2_t \rangle  $,
 for every $\varphi$ in   $(H\hat\otimes_\pi H)^*$
of the type $\varphi = j(a^* \otimes b^*)$. Since, by 
Lemma 4.17 of \cite{RusFab}, the algebraic tensor product
$H^* \otimes H^*$ is weakly-star  dense in $(H\hat\otimes_\pi H)^*$,
the uniqueness property $\U^1 = \U^2$ holds.
\item Suppose that $[\X, \Y]^\otimes$ exists. Let $a, b \in H$
and set $\varphi = j(a^* \otimes b^*)$. 
If $[\X, \Y]^\otimes $ exists, by \eqref{DZ1}, then
$ \langle \varphi, [\X, \Y]^\otimes \rangle =   [\langle \X, a \rangle,\langle \Y, b \rangle]$. 
 We  set now 
 $[\X,\Y]^{cl}(a, b) = [\langle \X, a \rangle,\langle \Y, b \rangle]$.  
By the usual isomorphism \eqref{Isomorphism}, between   $H\hat\otimes_\pi H$ and $\shl_1(H)$, 
according to the convention in Definition \ref{DZCovariation}, 
$a \mapsto [\X,\Y]^{cl}(a,\cdot)$ defines an $\shl_1(H)$-valued process. \\
\item If $\M$ and $\N$ are local martingales, then
         $\M$ and $\N$ admit a scalar quadratic because of 
Proposition 1.7 of \cite{DGR1}.
Moreover  $(\M,\N)$ admits a tensor 
covariation by Lemma 4.16 of \cite{RusFab}.
By previous item, it also admits an operator-valued
covariation.
\end{enumerate}
\end{proof}

We specify now our result for some particular Hilbert valued martingales, namely for the Brownian martingales. 
The framework is the same we used in Subsection \ref{DBZI}.

\begin{prop} \label{CovMart}
Let $U$ and $H$ be two separable real Hilbert spaces. Let $Q$ be a positive self-adjoint, injective operator in $\mathcal{L}(U)$. We set $U_0:=Q^{1/2} (U)$ and we consider $\W^Q=\{\W^Q_t:0\leq t\leq T\}$ an $U$-valued  $Q$-Wiener process with $\W^Q_0=0$, $\mathbb{P}$ a.s. 
Let us suppose that $(\mathscr F_t)$ is the canonical filtration
  generated by $\W^Q$ and consider a predictable  $\mathcal{L}_2(U_0,H)$-valued process $(\Phi_t)$ such that

\begin{equation} \label{F47}
\int_0^T {\mathrm Tr} [\Phi_r Q^{1/2} (\Phi_r Q^{1/2})^*]  d r   < \infty \qquad \mathbb{P}-a.s.
\end{equation}
and the process $\M$ defined as $\M_t = \int_0^t \Phi_r d{\W^Q}_r, t \in [0, T]$.
We have the following.
\begin{enumerate}
\item  $[\M,\M]^{cl}_t =  \int_0^t Q^\Phi_r dr$
where
\[
Q^\Phi_t = (\Phi_t Q^{1/2})(\Phi_t Q^{1/2})^*.
\]
\item  $[\M,\M]^{\otimes}_t$ is characterized by   
$ \langle j(a^*\otimes b^*),   [\M,\M]^{\otimes}_t \rangle
  = \int_0^t \langle a, Q^\Phi_s b \rangle ds$ 
for any $a,b \in H$.
\item For every $\varphi \in (H \hat\otimes_\pi H)^\ast$, we have
\begin{equation} \label{EECovMart}
  {}_{(H \hat\otimes_\pi H)^\ast} \left \langle \varphi,   
 [\M,\M]^{\otimes}_t \right \rangle_{H \hat\otimes_\pi H} =
\int_0^t {\mathrm Tr}(L_\varphi Q^\Phi_r)dr.
\end{equation} 
\end{enumerate}

 \end{prop}

\begin{proof}
\begin{enumerate} 
\item It is a consequence of  Theorem 4.12 in \cite{DaPratoZabczyk92}, 
where the result is stated under the hypothesis 
that the expectation of \eqref{F47} is finite.
 It can be extended to the general case with a stopping argument.
\item 
From Proposition \ref{PChainRule-2}
 we know that $\M$ is an $H$-valued local martingale. So
 by item 3. of Proposition \ref{PDZTens}, $\M$ admits both a  tensor and 
 an operator-valued quadratic variation; thanks to item 2. of the same proposition, they coincide 
once we have identified $H \hat \otimes_\pi H$ with $\shl^1(H)$. Item 2. of Proposition \ref{PDZTens} describes also the relation
 between the two and gives the evaluations $[\M,\M]^{\otimes}_t(a^*\otimes b^*)$.
 Since, by  Lemma 4.17 of \cite{RusFab}, the algebraic tensor product $H^* \otimes H^*$ is weakly-star  dense in
 $(H\hat\otimes_\pi H)^*$ the evaluation of  $[\M,\M]^{\otimes}_t$
on $j(a^*\otimes b^*)$ characterizes $[\M,\M]^{\otimes}_t$.
\item It follows from Proposition \ref{prop:26} and Proposition \ref{RExchange}.
\end{enumerate}
\end{proof}

\begin{lem} \label{L49}
Under the same assumptions of Proposition \ref{CovMart} we consider an
 $\mathcal{L}(H)$-valued process $(\Y_t)$ such that 
\begin{equation}
\int_0^T {\mathrm Tr} \left ( \Y_r \Phi_r Q^{1/2} ( \Y_r \Phi_r  Q^{1/2})^* 
\right ) d r < \infty  \qquad \mathbb{P}-a.s.
\end{equation}

Denote by $\mathbb{J}_t$ the element of ${(H\hat{\otimes}_{\pi}H)^{\ast}}$ corresponding to $\mathbb{Y}_t$ through
 the isomorphism described in (\ref{Isomorphism}).
Then
\[
\int_{0}^{t}
{}_{(H\hat{\otimes}_{\pi}H)^{\ast}}{\langle} \mathbb{J}_r, 
d[\mathbb{M}, \M]^\otimes_{r}\rangle_{H\hat{\otimes}_{\pi}H}dr
 = \int_{0}^{t}
{\mathrm Tr} \left [ Y_r (\Phi_r Q^{1/2})(\Phi_r Q^{1/2})^*  \right ]dr.
\]

\end{lem}
\begin{proof}
It is a consequence of
% Proposition XXX of \cite{DGR} 
 item 3. of Proposition \ref{CovMart}  and by Lemma \ref{LCovMart} below.
\end{proof}

\begin{lem} \label{LCovMart} 
Consider $L$ and $T$  in the sense of the lines before Proposition 
\ref{prop:26}.
 Let $\dot G: [0,T] \rightarrow \shl(H)$, and  $\dot g: 
 [0,T] \rightarrow (H \hat \otimes_\pi H)^\ast$
such that for every $r \in [0,T]$,  $\dot G(r) = L_{\dot g(r)}$
is Lebesgue-Bochner integrable on $[0,T]$. We define $G:  [0,T] \rightarrow \shl(H)$,
by $G(t) = \int_0^t \dot G(r) dr$ and  $g(t) = \int_0^t \dot g(r) dr$.
Let  $J:[0,T] \rightarrow  \shl_1(H)$ and $j:  [0,T] \rightarrow 
H \hat \otimes_\pi H$
such that for every $r \in [0,T], J(r) = T_{j(r)}$.   

If 
% $\int_0^T {\mathrm Tr} \left (J(r) \dot G(r) \right)  dr < \infty$,
$\int_0^T  \Vert J(r) \dot G(r) \Vert_{\shl_1(H)}  dr < \infty$,
 then 
 $$ 
\int_{0}^{t}
{}_{(H\hat{\otimes}_{\pi}H)^{\ast}} \langle \dot g(r),  j(r) 
\rangle_{H\hat{\otimes}_{\pi} H} dr
 = \int_{0}^{t}
  {\mathrm Tr} \left (\dot G(r) J(r) \right)  dr < \infty. $$
\end{lem}
\begin{proof} The proof follows first showing the equality for step functions
 $j$ (resp. $J$), and then passing to  the limit.
\end{proof}

\section{Notion of $\chi$-covariation}

\label{SecChi6}

\subsection{Basic definitions}

We introduce now a more general notion of covariation 
(and quadratic variation) than the
ones discussed before, which are essentially only suitable for
semimartingale processes.
The basic concepts were introduced in \cite{DGR, DGR1, DGR2}. 
The notion of  $\chi$-quadratic variation and 
 $\chi$-covariation
 is based on the notion of 
Chi-subspace.
Let $B, B_1, B_2$ be  separable Banach spaces.

\medskip 
\begin{dfn}\label{DefChi}
A Banach subspace $\chi$  continuously injected into  
$(B_1\hat{\otimes}_{\pi}B_2)^{\ast}$ will be called 
 {\bf  Chi-subspace} (of $(B_1\hat{\otimes}_{\pi}B_2)^{\ast}$).
In particular it holds 
\begin{equation} \label{EChiSpace}
\|\cdot\|_{\chi}\geq \|\cdot\|_{ (B_1\hat{\otimes}_{\pi}B_2)^{\ast}.} 
\end{equation}
\end{dfn}
Typical examples  of Chi-subspaces are the following.
\begin{enumerate}
\item Let  $\nu_1$ (resp. $\nu_2$)  be a dense subspace of $B_1^*$
(resp. $B_2^*$)
 then 
a typical Chi-subspace (of $(B_1 \hat \otimes_\pi B_2)^\ast$) 
 is  the topological projective tensor product
of $\nu_1  $ with    $\nu_2  $, denoted by  $\nu_1 \hat \otimes_\pi \nu_2$.
This is naturally embedded in  $(B_1\hat{\otimes}_{\pi}B_2)^{\ast}$ as recalled in Lemma \ref{lm:aggiunto}.
\item In particular, if $\nu_0$ is dense subspace of 
$B^\ast$, then $\chi := \nu_0 \hat \otimes_\pi \R$ is a Chi-subspace of 
$(B \hat \otimes_\pi \R)^\ast$, which can be naturally identified with
$B^\ast$. By a slight abuse of notations 
one could say that $\nu_0$ is a Chi-subspace of $B^\ast$.
\item Let $B$ be a separable Hilbert space $H$  and $A$  a generator 
of a $C_0$-semigroup on $H$, see \cite{EngelNagel99} and
  \cite{pa} Chapter 1 for a complete treatment of the subject.
 denote by $D(A)$ and $D(A^*)$ respectively the domains of $A$ and 
$A^*$ endowed with the graph norm, see again \cite{pa} Chapter 1
  or \cite{EngelNagel99} Chapter II.
Then a typical Chi-subspace of $(H \hat \otimes_\pi H)^\ast$ can 
be obtained setting 
$\chi := \nu_0 \hat \otimes_\pi \nu_0 $ and
$\nu_0 = D(A^{\ast})$ endowed with its the graph norm.
\item If $B = C([-\tau, 0])$, then 
$\chi$ could be the space $\shm([-\tau,0]^2)$
of finite signed measures on  $[-\tau,0]^2$.
Other examples of $\chi$-subspaces are given in Section \ref{window}.
\item It is not difficult to see that a direct sum of Chi-subspaces is a 
 Chi-subspace. This produces further examples of Chi-subspaces,
see Proposition 3.16 of \cite{DGR1}.
\end{enumerate}

Let $\X$ be a $B_1$-valued and $\Y$ be a $B_2$-valued process.
We suppose $\X$ to be continuous. Let $\chi$
be a Chi-subspace of $(B_1 \hat \otimes B_2)^\ast$.
We denote by $\shc([0,T])$ the space of real continuous processes equipped
 with the ucp topology. 
If $\varepsilon > 0$, we denote by $[\X, \Y]^{\epsilon}$ be the
 application
$$
[\X, \Y]^{\epsilon}: \chi\longrightarrow \shc([0,T])
$$ 
defined by
$$
\varphi
\mapsto
\left( \int_{0}^{t} 
%\prescript{}
{}_{\chi}{\langle} \varphi,
\frac{
J\left( 
 \left(\mathbb{X}_{r+\epsilon}-\mathbb{X}_{r}\right)\otimes
   \left(\mathbb{Y}_{r+\epsilon}-\mathbb{Y}_{r}\right)
\right)
}{\epsilon} 
\rangle_{\chi^{\ast}} \,dr 
\right)_{t\in [0,T]},$$
where $J:B_1 \hat{\otimes}_{\pi} B_2 \rightarrow
(B_1  \hat{\otimes}_{\pi}B_2)^{\ast\ast}$ is the canonical injection between a
 Banach space and its bidual (omitted  in the sequel). 

\begin{dfn}  \label{DChiCov}			
($\mathbb{X},\mathbb{Y})$
 {\bf admits a} \textbf{$\chi$-covariation} if the following holds. 
\begin{itemize}
\item[(H1)] For all $(\epsilon_{n})\downarrow 0$ it exists a subsequence 
$(\epsilon_{n_{k}})$ such that 
\[
\sup_{k}  \int_{0}^{T}  \frac{ \left\|  (\mathbb{X}_{r+\epsilon_{n_{k}}}-\mathbb{X}_{r})\otimes 
(\mathbb{Y}_{r+\epsilon_{n_{k}}}-\mathbb{Y}_{r})
  \right\|_{\chi^{\ast}}}
{\epsilon_{n_{k} } }dr \quad < \infty  \quad a.s.
\]
\item[(H2)] There exists a process, denoted by $[\mathbb{X}, \mathbb{Y}]^\chi:
 \chi \longrightarrow
  \shc([0,T])$ such that 
\[
[\X, \Y]^{\epsilon}(\varphi)\xrightarrow[\epsilon\rightarrow 0]{ucp} 
[\X, \Y](\varphi), \quad \forall \; \varphi\in \chi. 
\] 
%\end{itemize}

\item [(H3)]
There is a $\chi^{\ast}$-valued bounded variation process $\widetilde{[\X, \Y]^\chi}: [0,T] \times \Omega \rightarrow \chi^\ast$, such that
$\widetilde{[\X, \Y]^\chi}_{t}(\phi)=[\X, \Y]^\chi(\phi)_{t}, \forall t \in [0,T]$ a.s.
 for all
 $\phi\in \chi$.\\ 
\end{itemize}
\end{dfn}
\begin{dfn} \label{D53}
 If $B = B_1 = B_2$, and $\X = \Y$, we say that $\X$ has a
 $\chi$-{\bf quadratic 
variation}, if 
$(\X,\X)$ admits a  $\chi$-covariation.
\end{dfn}

\begin{dfn} \label{D54}
When $(\mathbb{X},\mathbb{Y})$  admits  a $\chi$-covariation, the 
 $\chi^{\ast}$-valued process
 $\widetilde{[\mathbb{X}, \mathbb{Y}]}$   (which is indeed
a modification of $[\mathbb{X}, \mathbb{Y}]$)
%(and even the application $[\mathbb{X}]$)
 will be called $\chi$-{\bf covariation} of $(\mathbb{X},\Y)$.
If    $\mathbb{X}$ admits a quadratic variation, the  $\chi^{\ast}$-valued 
process $\widetilde{[\mathbb{X},\X]}$, also denoted by
 $\widetilde{[\mathbb{X}]}$,
 is called  $\chi$-{\bf quadratic variation} 
of $\X$.
\end{dfn}

 \begin{rem} \label{R55}
\begin{enumerate}
\item $\widetilde{[\mathbb{X}]^\chi}$ will be the quadratic variation 
intervening in the second order derivative term of
 It\^o's formula stated in Theorem \ref{TITOF}, which will make
formula \eqref{FITOAtt} rigorous.
\item
For every fixed $\phi\in \chi$, the real processes 
$((\widetilde{[\X,\Y]^\chi})(\phi), t \in [0,T])$ and
 $([\X,\Y]^\chi(\phi)_{t}, t \in [0,T]), $
 are indistinguishable. 
\item
The $\chi^{\ast}$-valued process $\widetilde{[\X,\Y]}$ is weakly star 
continuous, i.e. 
$\widetilde{[\X]}(\phi)$ is continuous for every fixed $\phi\in \chi$,
 see \cite{DGR1} Remark 3.10 1.
\end{enumerate}
\end{rem}

A particular situation arises when 
 $\chi = (B_1 \hat{\otimes}_{\pi} B_2)^{\ast}$.
 
%{\bf Global quadratic variation concept}
\begin{dfn} \label{D56}
\begin{itemize}
\item We say that $(\mathbb{X},\mathbb{Y})$  admits a 
\textbf{ global covariation}
  if it admits a $\chi$-covariation with
 $\chi=(B_1 \hat{\otimes}_{\pi} B_2)^{\ast}$.
In this case we will omit the mention $\chi$ 
in  $\widetilde{[\mathbb{X},\Y]^\chi}$ and
$[\mathbb{X},\Y]^\chi$.
\item The modification
 $\widetilde{[\mathbb{X},\X]}$, which is a 
$(B\hat{\otimes}_{\pi} B)^{\ast\ast}$-valued process is also called 
{\bf global quadratic variation} of $\X$.

\end{itemize}
\end{dfn}
\begin{rem}\label{CClassGlobal}
The following statements are easy to establish, see Remarks
4.8 and 4.10 of \cite{RusFab}.
\begin{enumerate} 
\item If $\X$ has zero scalar quadratic variation 
then $\X$ has a zero tensor  quadratic variation and 
 $\X$ has a zero global quadratic variation.
\item If  $\X$ and $\Y$ have a scalar quadratic variation and 
$(\X,\Y)$ has  a tensor covariation,
then $(\X,\Y)$ admit a global covariation.
Moreover $\widetilde {[\X, \Y]} = [\X, \Y]^\otimes$,
where the equality holds in $B_1 \hat \otimes_\pi B_2$.
\item If $(\X,\Y)$ admits a global covariation, then
 it admits a $\chi$-covariation for every Chi-subspace $\chi$.
Moreover  $ \widetilde {[\X,\Y]_t^\chi}(\varphi) =  \widetilde{[\X,\Y]_t}(\varphi)$,
for every $t \in [0,T], \varphi \in \chi$.
\end{enumerate}
\end{rem}
\begin{prop} \label{PClassGlobal}
Let $\X^i = \M^i + \V^i, i =1,2$ be two semimartingales
with values in $B_i$.
Let $\chi$ any Chi-subspace of $(B_1 \hat \otimes_\pi B_2)^\ast$.
Then $(\X^1, \X^2)$ admits a $\chi$-covariation and \\
$\widetilde {[\X^1,\X^2]^\chi}(\varphi) =  {}_{H \hat \otimes_\pi H} \langle 
[\M^1,\M^2]^{\otimes}, \varphi \rangle_{(H \hat \otimes_\pi H)^\ast}, \ \forall \varphi \in \chi$.
\end{prop}
\begin{proof} \
By item 2. of Remark \ref{R7.7}, $\V$ has a zero scalar quadratic variation.
By Proposition \ref{PDZTens} 3. $(\M_1, \M_2)$ admits a tensor
covariation. By item 1. of Remark \ref{R7.7}
and by the linearity of  tensor covariation 
it follows that  $[\X^1,\X^2]^\otimes = [\M^1,\M^2]^\otimes$.
Again by point 1. of   Remark \ref{R7.7}, $\X^1$ and $\X^2$ have a scalar
 quadratic variation. 
Again by Remark  \ref{CClassGlobal}  2., $(\X^1,\X^2)$ admits a global
 quadratic variation and so the result follows by Remark  \ref{CClassGlobal} 3.
%
%\ref{CClassGlobal}.
\end{proof}

 Indeed  the notion of global covariation is closely related to the 
weak-$\ast$  convergence 
in $(B\hat{\otimes}_{\pi} B)^{\ast\ast}$.
If the probability space $\Omega$ were a singleton,
i.e. in the deterministic case,
if $\X$  admits a $\chi$-quadratic variation then 
  $$[\X, \X]^{\epsilon}_t \xrightarrow 
[\epsilon\rightarrow 0]{w^{\ast}}\widetilde{[\X, \X]_t}, 
\forall t \in [0,T]. $$

As we mentioned, the notion of weak Dirichlet process 
admits a generalization to the Banach space case.
 \begin{dfn}
 \label{def:nu-dir}
Let $\V, \X$ be two $B$-valued continuous processes and $\nu_0$ be a dense subspace of $B^\ast$.
We set $\nu = \nu_0 \otimes \R$.
\begin{enumerate}
\item $\V$ is said $({\mathscr F}_t)-\nu$-{\bf martingale orthogonal process} if for any real $({\mathscr F}_t)$-local martingale $N$
we have $[\V, N]^\nu = 0$
\item $\X$ is said  $({\mathscr F}_t)-\nu$-{\bf weak Dirichlet} if it is the sum of an 
a $({\mathscr F}_t)$-local martingale $\M$ and an  $({\mathscr F}_t)-\nu$-martingale orthogonal process.
\end{enumerate}
\end{dfn}
\begin{rem} \label{R510}
\begin{enumerate}
\item If $B = \R$, then any  $({\mathscr F}_t)-\nu$-weak Dirichlet (resp.  $({\mathscr F}_t)-\nu$-martingale orthogonal) 
process is a real $({\mathscr F}_t)$-weak Dirichlet (resp. $({\mathscr F}_t)$-martingale orthogonal) process.
\item The notions introduced in Definition \ref{def:nu-dir} 
%of Dirichlet, weak Dirichlet, $\nu$-weak Dirichlet process
depend on an underlying filtration $({\mathscr F}_t)$. When not necessary
it will be omitted. We will speak about 
Dirichlet  (resp. weak Dirichlet,  $\nu$-weak Dirichlet process)
 instead of  $({\mathscr F}_t)$-Dirichlet (resp. $({\mathscr F}_t)$-weak Dirichlet, 
 $({\mathscr F}_t)$-$\nu$-weak
 Dirichlet process).
\end{enumerate}
\end{rem}
\begin{rem} \label{R511} 
Let $H$ be a separable Hilbert space and $\nu_0$ be a Banach space continuously
embedded in $H^\ast$. We set $\chi = \nu_0 \hat \otimes_\pi \nu_0, \
\nu = \nu_0 \otimes \R$. A zero $\chi$-quadratic variation process
is a $\nu$-weak orthogonal process.
This was the object of Proposition 4.29 in \cite{RusFab}.
\end{rem}

We introduce below the  useful notion of $\bar \nu_0$-semimartingale.
\begin{dfn} \label{nu-semi} 
Let $(\S_t, t \in [0,T])$ be an $H$-valued progressively measurable process
and a Banach space $\bar \nu_0 $  in which $H$ is continuously embedded.
$\S$ is said $\bar \nu_0$-{\bf semimartingale} (or more precisely 
$\bar \nu_0-({\mathscr F}_t)$-semimartingale)
if it is the sum  of a local martingale $\M$ and a process $\A$ 
which  finite variation as $\bar \nu_0$-valued process. 
\end{dfn}

\begin{prop} \label{P65}
\begin{enumerate}
\item An $H$-valued $\bar \nu_0$-semimartingale is a 
semimartingale as $\bar \nu_0$-valued process. 
\item The decomposition of a $\bar \nu_0$-semimartingale is unique, if for instance we prescribe
that $\A_0 = 0$ a.s.
\end{enumerate}
\end{prop}
\begin{proof} \
\begin{enumerate}
\item Indeed an $H$-valued martingale is clearly a $\bar \nu_0$-valued 
martingale and consequently, by stopping arguments, an  $H$-valued local martingale 
is a $\bar \nu_0$-semimartingale.
\item It follows by the decomposition of a semimartingale taking values in $\bar \nu_0$.
\end{enumerate}
\end{proof}
The uniqueness  of the decomposition of a $\bar \nu_0$-semimartingale allows to define an extension
of It\^o integral, that will still denoted in the same way.
\begin{dfn} \label{D66}
Let $H, E$ be separable Hilbert spaces. Let $\bar \nu_0$ be a Banach space in which $H$ is continuously
injected and $\S = \M + \A$ be a  $H$-valued which is a
$\bar \nu_0$-semimartingale.
Suppose that $(\Y_t)$ is a progressively measurable, such that 
\begin{equation} \label{Etech1}
\int_0^T  \Vert \Y_r \Vert_{\shl(H,E)}^2 d [\M]^{\R, cl}_r    + 
\int_0^T  \Vert \Y_r \Vert_{\shl(\bar \nu_0,E)} d\Vert A \Vert_r
  < \infty,
\end{equation}
where $r \mapsto \Vert \A(r) \Vert$ is the total variation function
 of $r \mapsto \A(r)$.
We denote by $\int_0^t \Y_s d\S_s := \int_0^t \Y_s d\M_s +
  \int_0^t \Y_s d\A_s, t \in [0,T]$. 
\end{dfn}

\begin{prop}\label{P67} Let $H$ be a separable Hilbert space, 
continuously embedded
in a Banach space $\bar \nu_0$. Let $\S = \M + \A $ 
be an $H$-valued process which is a $\bar \nu_0$-semimartingale.
We set $\nu_0 =  \bar \nu_0^\ast$. We set $\chi = \nu_0 \hat \otimes_\pi \nu_0$.
\begin{enumerate}
\item $\A$ admits a zero $\chi$-quadratic variation.
\item $[\M,\A]^\chi = 0$.
\item $\S$ is a $\nu_0 \hat \otimes_\pi \R $-weak Dirichlet process.
\item  $\S$ has a 
%$\nu_0 \hat \otimes_\pi \nu_0$
$\chi$-quadratic variation.
Moreover $\widetilde{[\S, \S]}^\chi (\varphi) = 
{}_{(H \hat \otimes_\pi H)^\ast} \langle \varphi, [\M,\M]^\otimes 
\rangle_{H \hat \otimes_\pi H}$, $\forall \varphi \in  (H \hat \otimes_\pi H)^\ast$.
\end{enumerate}

\end{prop}
\begin{proof} 
\begin{enumerate}
\item Observe that, thanks to Lemma 3.18 in \cite{DGR1},
 it will be enough
 to show that 
\begin{equation} \label{E111}
I(\epsilon):= \frac{1}{\epsilon}\int_0^T 
|(\A(r+\epsilon) - \A(r))\otimes^2|_{\chi^*}
  dr \xrightarrow{\epsilon \to 0}0, \qquad \text{in probability}.
\end{equation}
In fact, identifying $\chi^*$ with the space of bounded
bilinear functions on $\nu_0$, i.e. ${\mathcal Bi}(\nu_0, \nu_0)$,
recalling that $\nu_0 = {\bar \nu}_0^\ast$,
the left-hand side  of \eqref{E111} gives
\begin{eqnarray*}
I(\epsilon) &=& \frac{1}{\epsilon} \int_0^T
 \sup_{|\phi|_{\bar{\nu_0}},\,
 |\psi|_{\bar{\nu_0}} \leq 1} \left | \left\langle (\A(r+\epsilon) - 
\A(r)), \phi \right \rangle \left\langle (\A(r+\epsilon) - \A(r)),
 \psi \right \rangle \right | d r \\
&\le & \frac{1}{\varepsilon} \int_0^T \vert \A(r+\epsilon) - \A(r) \vert^2_{\bar \nu_0} dr.
\end{eqnarray*}
Since $\A$ is an $\bar \nu_0$-valued bounded variation process,
previous quantity converges to zero, by Remark \ref{R7.7} 2.
\item It follows by very close arguments. In particular, an adaptation of Lemma 3.18 of \cite{DGR1}
shows that it will be enough to show that
\begin{equation} \label{E111bis}
J(\epsilon):= \frac{1}{\epsilon}\int_0^T 
|(\A(r+\epsilon) - \A(r))\otimes\M(r+\epsilon) - \M(r))  |_{\chi^*}
  dr \xrightarrow{\epsilon \to 0}0, \qquad \text{in probability}.
\end{equation}
Then we use the fact that $\M$ is a $\bar \nu_0$-valued  local martingale and therefore,
by item 3. of Proposition \ref{PDZTens}, it has a scalar quadratic variation,
as $\bar \nu_0$-valued process.
\item follows by Remark \ref{R511}.
\item 
Indeed the bilinearity of the $\chi$-covariation and items 1. and 2.
imply that $[\S,\S]^\chi = [\M,\M]^\chi$. The result follows then
 by Proposition \ref{PClassGlobal}.
\end{enumerate}

\end{proof}

Below we will state  examples of processes having a
 $\chi$-quadratic variation.

\subsection{Window processes}
\label{window}

Let $B = C([-\tau,0])$, for some $\tau > 0$, $X = (X_t, t\in [0,T])$
be a real process and $\X = (X_t(\cdot), t \in [0,T]),$ the corresponding
{\bf window process}, i.e. such that $X_t(x) = X_{t+x}, x \in [-\tau,0]$.
 We start with some basic examples.
\begin{prop} \label{P516}
If $X$ has H\"older continuous paths with parameter $\gamma > \frac{1}{2}$ then $X(\cdot)$
 has a zero scalar quadratic variation and therefore a global quadratic variation.
\end{prop}
\begin{proof}
It follows directly from the definition and the H\"older path property.
\end{proof} 
A typical example of such processes
 is fractional Brownian motion with Hurst parameter $H > \frac{1}{2}$
or the bifractional Brownian motion with parameters $H, K$ and $HK > \frac{1}{2}$, see for
 instance \cite{hv, rtudor}.
By Proposition 4.7 \cite{DGR1},  the window of a classical Wiener process 
has no scalar quadratic variation so no 
  global quadratic variation 
since condition (H1) in Definition \ref{DChiCov} cannot be fulfilled with respect to 
$\chi= (B\hat{\otimes}_{\pi}B)^*$.
For this reason, it is important to investigate if it has a  
$\chi$-quadratic variation for a suitable subspace $\chi$
of $(B\hat{\otimes}_{\pi}B)^{\ast}$. 
 In the framework of window processes, typical examples of $\chi$ are the following.
\begin{enumerate}
\item $\mathcal{M}([-T,0]^{2})$ equipped with the total variation norm. 
\item $L^{2}([-\tau,0]^{2})$.
\item $\mathcal{D}_{0,0}=\{ \mu(dx,dy)=\lambda \,  
\delta_{0}(dx)\otimes \delta_{0}(dy), \lambda \in \R \} $.
\item Let $\mathcal{D}_{0}$ be the one-dimensional space of measures obtained as 
multiple of the Dirac measure  $\delta_{0}$. 
The following linear subspace of  $\mathcal{M}([-T,0]^{2})$ given by
 $$  \mathcal{D}_{0,0} \; \oplus \; \Big( L^{2}([-T,0])  {\otimes}  \shd_{0} \Big) \; 
\oplus \; \Big( \shd_{0} {\otimes} L^{2}([-T,0]) \Big) \; \oplus \; L^{2}([-T,0]^{2}).$$
This is a Banach space, equipped with a self-explained sum of four norms.
By the lines above Remark 3.5 in \cite{DGR1}, that space is the Hilbert tensor
product  $(\shd_{0}\oplus L^{2})\hat{\otimes}_{h}^{2}$.
\item $Diag:=\left\{\mu(dx,dy)=g(x)\delta_{y}(dx)dy; g \in L^{\infty}([-T,0])
\right\}$.
\item The direct sum $\chi_0$ of the spaces defined in  4. and  5.
 is a Chi-subspace. We remind item 5. at the beginning of Section \ref{window}.
\end{enumerate}
\begin{rem} \label{R518Chi}
The window Brownian
motion $W(\cdot)$ does not have a
$\chi$ -quadratic variation for  $\chi = \mathcal{M}([-\tau,0]^{2})$.
This follows because  the bidual of $C([-\tau,0]^2)$ is isometrically
embedded into its bidual, and  the window Brownian motion has no scalar
quadratic variation. In particular condition (H1) of the $\chi$-covariation cannot be fulfilled.
\end{rem}
In all the other cases a classical Wiener process has a $\chi$-quadratic
variation. Indeed this also extends to the case of
a generic finite quadratic variation process.
The proposition below is the consequence of 
Propositions 4.8 and 4.15 of \cite{DGR1}
and the fact that the direct sum of Chi-subspaces is a Chi-subspace.

From now on, in this section, for simplicity we set $\tau = T$.
\begin{prop} \label{PChiCalculation}
 Let $X$ be a finite quadratic variation process. Then $X(\cdot)$ has 
a $\chi_0$-quadratic variation.
Moreover, for $\mu \in \chi_0$, we have
$$[X(\cdot)]_{t}(\mu)=\int_{D_{t}} d\mu(x,y) [X]_{t-x},$$
where $D_{t}$ is the diagonal
$ \{(x,y) \in  [-T,0]^2 \vert -t \le x = y \le 0\}$.
\end{prop}

In fact a Chi-subspace will plays the role of a {\it suitable}
subspace of $ (B \hat \otimes_\pi B)^\ast$,
  in which lives the second Fr\'echet derivative 
of a functional $F:B \rightarrow \R$ is forced to live,
 in view of expanding $F(\X)$ via a It\^o type 
formula of the type \eqref{FITOAtt}.

\begin{ese} \label{EXE}
Here are some typical particular cases of elementary functionals
whose second derivatives belong to some Chi-spaces mentioned above.
The details of the verification are left to the reader.
\begin{enumerate}
\item $F(\eta)=f(\eta(0))$ where $f: \R \rightarrow \R$
is of class $C^2$. Then 
 $D^{2}F(\eta) \in \shd_{0,0}$ for every $\eta \in B$.
\item $F(\eta)=\left( \int_{-T}^{0}\eta(s)ds \right)^{2}$. Then,
 for every $\eta \in B$,
   $D^{2} F(\eta) \in (\shd_{0}\oplus L^{2})\hat{\otimes}_{h}^{2}$.
\item $F(\eta)=\int_{-T}^{0} \eta(s)^{2}ds$. In this case 
  $D^{2} F(\eta) \in Diag$
%\oplus \mathcal{D}_{0,0}$ 
for  every $\eta \in B$.
\end{enumerate}
\end{ese}

\subsection{Convolution type processes}
\label{Conv}

Let $H$ be a separable Hilbert space.
Those processes, taking values in $H$, are the natural generalization of It\^o type processes.
Let $A$ be the generator of a $C_0$-semigroup on $H$. 
 Denote again by $D(A)$ and $D(A^*)$ respectively the domains of $A$ and 
$A^*$ endowed with the graph norm.

Let $U_0, U$ be separable Hilbert
spaces
according to Sections \ref{SGPF} and \ref{DBZI}.
Let  $\W$ be a $Q$-Wiener process with values in $U$ where $Q \in \mathcal{L}(U)$ a positive, injective 
and self-adjoint operator and define again $U_0:=Q^{1/2} (U)$
 endowed with the scalar product $\left\langle a,b \right\rangle_{U_0} := \left\langle Q^{-1/2}a,Q^{-1/2}b
 \right\rangle$. 
 Let ${\sigma} =  (\sigma_t, t \in [0,T])$
  with paths a.s. in  $\shl_2(U_0;H)$ and $b =  (b_t, t \in [0,T])$
with paths taking values in $H$ 
being                                 
predictable
such that 
\begin{equation}
\label{eq:hp-per-existence-mild}
\mathbb{P} \left [ \int_0^T \left (
 \|\sigma_t\|^2_{\mathcal{L}_2(U_0; H)} + |b_t| \right)  dt < \infty \right ] = 1.
\end{equation}
Let $x_0 \in H$. 

\begin{dfn} \label{D314} A continuous process of the type 
\begin{equation} \label{DConvProc}
 \X_t = e^{tA} x_0 + \int_0^t e^{(t-r)A} \sigma_r d \W_r +
 \int_0^t e^{(t-r)A} b_r dr,
\end{equation}
is said  {\bf convolution type process} (related to $A$).
\end{dfn}
$e^{tA}$ stands of course for the $C_0$-semigroup associated
with $A$. 
Clearly if $A = 0$, the semigroup is the identity, 
so a convolution type process is
an It\^o type process.
Natural examples of convolution processes are given by mild 
solutions of  stochastic PDEs, see for instance \cite{DaPratoZabczyk92} Chapter 7 
or \cite{GawareckiMandrekar10} Chapter 3.1.
\begin{prop}\label{PConv}
 Let $\X$ be a convolution type process as 
in \eqref{DConvProc}
and 
$$\nu_0 = D(A^\ast) \subset H^\ast, \quad \chi = \nu_0 \hat \otimes_\pi \nu_0.$$
The following properties hold.
\begin{enumerate}
\item 
 $\X$  admits a decomposition of the type
$\M + \V$
where
$$
\M_t = x_0 + \int_0^t\sigma_r d\W_r, \quad
\V_t  = \int_0^t b_r dr + \A_t, \ t \in [0,T],
$$
where $\A$ is a progressively measurable process such that
\begin{equation}
\label{eq:properA}
{}_{H}\langle \A_t, \phi \rangle_{H^\ast}  = \int_0^t 
   {}_{H}\langle \X_r, A^\ast  \phi \rangle_{H^\ast} dr,
 \  \forall  \phi \in \nu_0.
\end{equation}
\item Let $\bar \nu_0$ be the dual of $D(A^\ast)$, $\bar \nu_0$ contains $H$ since $D(A^\ast)$ and $H$ are Hilbert spaces and then reflexive.
Then $\X$ is an $\bar \nu_0$-semimartingale with decomposition $\M + \V$,
$\M$ being the local martingale part.
\item The process $\A$ appearing in 1. admits a $\chi$-zero
quadratic variation.
\item $\X$ admits a $\chi$-quadratic variation given by
\begin{equation} \label{EE3}
\widetilde{[\X,\X]^\chi}(\varphi) =  \int_0^t 
{\mathrm Tr}\left(L_\varphi  (\sigma_r 
Q^{\frac{1}{2}})  (\mathbb \sigma_r Q^{\frac{1}{2}} )^\ast \right) dr,\quad \varphi \in \chi , 
\end{equation}
where $L_\varphi$ was defined in \eqref{eq:expressionLB}.
\end{enumerate}
\end{prop}
\begin{proof} 

\begin{enumerate}
\item This follows by Theorem 12, \cite{Ondrejat04}, see also
 Lemma 5.1 \cite{RusFab}.
% Corollary 5.3 of \cite{RusFab}.
\item 
The process $\A$ can be considered as a $\bar\nu_0$-valued process. From
\eqref{eq:properA}, it follows that, for $0\leq s \leq t \leq T$ and $\phi\in D(A^*)$,
 using (\ref{eq:properA}), we have
\[
\phantom{.}_H \left\langle \A_t - \A_s , \phi \right\rangle_{H^*} = \int_s^t d s \phantom{.}_H\! 
\left\langle \X_r, A^* \phi \right\rangle_{H^*},
\]
so the $|\cdot|_{\bar \nu_0}$ norm of $\A_t - \A_s$ is estimated by 
\[
\sup_{\begin{subarray}{l}
       \phi\in D(A^*)\\
       |\phi|_{D(A^*)} \leq 1
      \end{subarray} }
      | \left\langle \A_t - \A_s , \phi \right\rangle | \leq 
\int_s^t d s 
\sup_{\begin{subarray}{l}
       \phi\in D(A^*)\\
       |\phi|_{D(A^*)} \leq 1
      \end{subarray} }
\phantom{.}_H \vert \left\langle \X_r, A^* \phi \right\rangle_{H^*} \vert
\leq
\int_s^t  
\vert \X_r \vert_H d r.
\]

Previous inequalities show that $\A$ has a total variation as $\bar\nu_0$-valued process which is bounded by 
$
\int_0^T dr |\X_r|_H.
$
Since the $\bar\nu_0$-norm is dominated by  the $H$-norm and
 $\int_0^\cdot b_r dr$ is an $H$-valued bounded 
variation process, then $\V$ is also a bounded variation $\bar\nu_0$-valued process. 
Finally $\X$ is a $\bar\nu_0$-semimartingale.
\item follows from item 1. of Proposition \ref{P67}.
\item follows from item 4. of Proposition \ref{P67} and item 3. of
 Proposition \ref{CovMart}.
%OPERATORIALE
\end{enumerate}
\end{proof}

\section{Stochastic calculus}

\subsection{Banach space valued forward integrals}

Let $U, H$ be separable Hilbert spaces and $B, E$ be separable Banach spaces.

\begin{dfn} \label{ForwardInt}
 Let $(\Y_t, t \in [0,T])$
be a (strongly) measurable process taking values in $\shl(B,E)$ and $\X = (\X_t, t \in [0,T])$,
be a $B$-valued continuous process and the following.
$\int_0^T \Vert \Y_r \Vert_{\shl(B,E)} dr < \infty$. a.s.
We suppose the following.
\begin{itemize}
\item $\lim_{\varepsilon \rightarrow 0} \int_0^t \Y_r \frac{\X_{r+\varepsilon}
- \X_r}{\varepsilon} dr$ exists in probability for any $t \in [0,T]$.
\item Previous limit random function admits a continuous version.
\end{itemize}
In this case, we say that the {\bf forward integral} 
of $\Y$ with respect to $\X$, denoted by
$\int_0^\cdot \Y d^-\X$ exists.
\end{dfn}
\begin{rem} \label{RScalarCase}
\begin{enumerate}
\item
If $E = \R$ than we often denote 
$ \int_0^\cdot \Y_r d\X_r = \int_0^\cdot 
 {}_{B^\ast} \langle \Y_r, d\X_r \rangle_B.$  
\item If $\X = \V$ is a continuous  bounded
variation process, and $\Y$ is an a.s. bounded  measurable
process having at most  countable number  jumps
(as for instance cadlag or caglad), then  
$ \int_0^\cdot \Y_r d^- \X_r$ exists and it equals
the Bochner-Lebesgue integral $ \int_0^\cdot \Y_r d\X_r$. \\
If $\X$ is a.s. is absolutely continuous with derivative
$r \mapsto {\dot X_r}$, then, whenever 
$\int_0^T \Vert Y_r \Vert_{\shl(B,E)}
 \vert {\dot X}_r \vert_B  dr < \infty$ a.s., 
 $ \int_0^\cdot \Y_r d^-\X_r$ exists and 
it equals the same Bochner integral as before.\\
In both cases, the proof is similar to the case when 
the processes are real-valued, see e.g.
Proposition 1.1 in \cite{rv1} making use of stochastic Fubini's theorem,
i.e. Theorem 4.18 of \cite{DaPratoZabczyk92}.
\item Suppose that $B = H$ and $E = U$. Let 
$\X = \M$ be an $({\mathscr F}_t)$-local martingale and $\Y$ be a 
predictable process such that
$\int_0^T \Vert \Y_r \Vert^2_{\shl(U,H)} d[M]_r^{\R, cl} < \infty$
a.s. Then $\int_0^t \Y_r d^- \M_r$ exists and it equals 
the It\^o type integral  $\int_0^t \Y_r d\M_r$, see Theorem 3.6 in 
\cite{RusFab}.
\item Suppose that $\M = \W$ is a $Q$-Wiener process with values in a separable Hilbert space $H$
and $\Y$ is a predictable process such that
$\int_0^T  {\mathrm Tr} ((\Y_r Q^{\frac{1}{2}}) (\Y_r Q^{\frac{1}{2}})^\ast dr < \infty$ a.s.
Then the forward integral $\int_0^t \Y_r d ^-\W_r, t \in [0,T]$ exists
and it equals the It\^o integral  $\int_0^t \Y_r d \W_r,
  t \in [0,T]$, see Theorem 3.4 in 
\cite{RusFab}.
\item A consequence of the previous two items is the following.
If $\X = \M + \V$ is an $H$-valued semimartingale, and $\Y$ is a cadlag
predictable process, then
$\int_0^t \Y d^- \X, t \in [0,T]$, exists and it is the sum
$\int_0^t \Y d \M + \int_0^t \Y d \V$.
\end{enumerate}
\end{rem}

\subsection{It\^o formulae}

We can now state the following Banach space valued   It\^o's formula, see Theorem 5.2 of \cite{DGR1}.
 \medskip

\begin{thm} \label{TITOF}
 Let $B$ a separable Banach space, $\chi$ be a Chi-subspace of $(B\hat{\otimes}_{\pi}B)^{\ast}$ and 
let $\mathbb{X}$ a $B$-valued continuous process admitting a $\chi$-quadratic variation.
Let $F:[0,T]\times B\longrightarrow \mathbb{R}$ be $C^{1,2}$ Fr\'echet such that
$$
D^{2}F:[0,T]\times B\longrightarrow \chi \subset
(B\hat{\otimes}_{\pi}B)^{\ast} \quad \mbox{ continuously}. 
$$
Then  the forward integral 
\begin{equation} \label{FIFOF}
\int_{0}^{t}
%\prescript{}
{}_{B^{\ast}} {\langle} DF(s,\mathbb{X}_{s}),d^{-}\mathbb{X}_{s}\rangle_{B}, \ t \in [0,T],
\end{equation}
exists and the following formula holds: 
\begin{equation} \label{FIFO1}
%\begin{split}
F(t,\mathbb{X}_{t})
%&
=F(0, \mathbb{X}_{0})+\int_{0}^{t}\partial_{r}
F(r,\mathbb{X}_{r})dr  + \int_{0}^{t}
{}_{B^{\ast}} {\langle} DF(r,\mathbb{X}_{r}),d^{-}\mathbb{X}_{r}\rangle_{B} +\\
%&
\frac{1}{2}\int_{0}^{t} 
{}_{\chi} {\langle} D^{2}F(r,\mathbb{X}_{r}),d\widetilde{[\mathbb{X}]}_{r}
\rangle_{\chi^{\ast}}.\\
%\end{split}
\end{equation}
\end{thm} 
%\begin{rem} \label{RITOF}
 The assumption that  the second derivatives to lives in a suitable $\chi$-space can be relaxed
in some situations, see for instance Proposition \ref{P69} below. 
%\end{rem}

\begin{prop} \label{P69}
Let $H$ be a separable Hilbert space.
Let $\nu_0$ be a dense subset of $H^\ast$. We set  $\chi = \nu_0 \hat 
\otimes_\pi \nu_0$. 
Let $\X$ be a $\chi$-finite quadratic variation $H$-valued process. 
Let $F\colon [0,T] \times H \to \mathbb{R}$ of class $C^{1,2}$ such that
 $(t,x)\to D F(t,x)$ is continuous from $[0,T]\times H$ to $\nu_0$.
 Suppose moreover the following assumptions.
\begin{itemize}
\item[(i)] There exists a (cadlag) bounded variation process 
$C\colon [0,T] \times \Omega \to H\hat\otimes_\pi H$ such that, for all $t$ in $[0,T]$ and $\varphi\in \chi$, 
\[
{}_{H \hat \otimes_\pi H} \langle C_t(\cdot),\varphi \rangle_{(H \hat \otimes_\pi H)^\ast}
 = {[\X, \X]_t^{\chi}}(\varphi)(\cdot) \qquad a.s.
\]
\item[(ii)] For every continuous function $\Gamma\colon [0,T] \times H \to \nu_0$ the integral
\begin{equation} \label{ForwI}
\int_0^\cdot \left\langle \Gamma(r,\X_r), d ^- \X_r \right\rangle
\end{equation}
exists.
\end{itemize}
Then
\begin{eqnarray}
\label{eq:Ito}
F(t,\X_t) &=& F(0, \X_0) + \int_0^t {}_{H^\ast}\left\langle D F(r, \X_r), d^- \X_r \right\rangle_H 
\nonumber \\
&& \\
&+ & \frac12 \int_0^t  {}_{(H \hat\otimes_{\pi} H)^*}
\left \langle D^2 F(r, \X_r) , d C_r 
\right\rangle_{H \hat\otimes_{\pi} H} 
  + \int_0^t \partial_r F(r, \X_r) d r \nonumber.
\end{eqnarray}
\end{prop}
\begin{proof}
In Theorem 5.4 of \cite{RusFab}, the result is formulated for the
 particular case $\nu_0 = D(A^*)$ where $A$ is
 the generator of a $C_0$-semigroup;  the arguments
to extend the result to the case of a generic $\nu_0$ are the same.

\end{proof}
\begin{rem}
\label{R68}
Clearly 
$$
(t,\omega) \mapsto \left (\varphi \mapsto 
  {}_{\chi^\ast} \langle C_t(\omega), \varphi \rangle_\chi \right) 
= \widetilde {[\X,\X]^\chi}_t(\omega), 
$$
are indistinguishable processes with values in $\chi^\ast$,
if we identify $\chi^\ast$ as a space which contains the
bidual of $H \hat \otimes_\pi H$ and therefore $H \hat \otimes_\pi H$
itself.
\end{rem}
A consequence of previous proposition is a natural It\^o
formula for convolution type processes.

\begin{prop}
\label{lm:Ito}
Let $\X$ be a convolution type process as in Definition
\ref{D314} with $\sigma$ and $b$ verifying 
\eqref{eq:hp-per-existence-mild}.
Assume that $F \in C^{1,2}([0,T] \times H)$ with $D F \in C([0,T] \times 
H, D(A^*))$.
Then, for every $t \in [0,T]$,
\begin{multline}
\label{eq:first-Dinkin}
F(t,\X_t) =  F(0,\X_0)  + \int_0^t \partial_r 
F(r,\X_r) d r\\  
+  \int_0^t \left\langle D F(r,\X_r), b_r \right\rangle d r + \int_0^t
 \left\langle D F(r,\X_r), \sigma_r d \W^Q_r \right\rangle\\ 
+  \int_0^t \left\langle A^* D F(r,\X_r), \X_r \right\rangle d r
+ \frac{1}{2}  \int_0^t \text{Tr} \left [\left ( \sigma_r  {Q}^{1/2} \right )
\left ( \sigma_r Q^{1/2}  \right)^\ast D^2 F(r,\X_r) \right ] d r,
 \qquad \mathbb{P}-a.s.
\end{multline}
where for $(r,\eta) \in [0,T] \times H$,
again we associate   $D^2 F(r,\eta)$, which is in principle
an element of ${\mathcal Bi}(H,H)$, with a map in $\shl(H)$, as in
\eqref{eq:expressionLB}.
\end{prop}
\begin{proof}
It is a consequence of Proposition \ref{P69} using Proposition \ref{PConv}
as follow.
Let $\chi = \nu_0 \hat \otimes_\pi \nu_0$ with 
$ \nu_0 = D(A^\ast)$.

Indeed, thanks to item 4. of Proposition \ref{PConv}, $\X$ admits a
 $\chi$-quadratic variation.
Consider the decomposition $\M + \V$ defined in item 1. of Proposition \ref{PConv}.
We first check that hypothesis (ii) of Proposition \ref{P69} is satisfied. 
$\int_0^t \langle \Gamma(r, \X_r), d^- \M_r \rangle$
(resp. $\int_0^t \langle \Gamma(r, \X_r), d^- \int_0^\cdot b_r dr  \rangle$)
exists and it equals the It\^o type integral 
\begin{equation} \label{AA10}
\int_0^t \langle \Gamma(r, \X_r), d \M_r \rangle, \  t \in [0,T],
\end{equation}
(resp. 
\begin{equation} \label{AA11}
\int_0^t \langle \Gamma(r, \X_r),  b_r \rangle dr,  \  t\in [0,T]) .
\end{equation}
This happens because of items 2. and 3. of Remark \ref{RScalarCase}.
Consequently 
%\begin{equation} \label{AA12}
$\int_0^t \langle \Gamma(r, \X_r), d^- \X_r \rangle$
%\end{equation}
exists  if 
\begin{equation}
\label{ForwI-4}
\int_0^t \left\langle \Gamma(r,\X_r), d^- \A_r \right\rangle, \qquad t\in [0,T],
\end{equation}
exists and it equals the sum of \eqref{AA10}, \eqref{AA11}
and  \eqref{ForwI-4}. We recall that $\A$
was defined  in Proposition \ref{PConv}. \\
 So let  us show that \eqref{ForwI-4} exists.
For every $t\in [0,T]$, using (\ref{eq:properA}), we evaluate 
limit of its  $\epsilon$-approximation, taking into account item 1. of Proposition \ref{PConv}.
Denoting  $R(\epsilon,t)$
 a remainder boundary term which converges a.s. to zero, we get
\begin{multline}
\label{eq:era38}
\frac{1}{\epsilon}\int_0^t \left\langle \Gamma(r,\X_r), \A_{r+\epsilon} - 
\A_r \right\rangle d r
=  \frac{1}{\epsilon}\int_0^t \int_{r}^{r+\epsilon} \left\langle \X_u, A^*\Gamma(r,\X_r) \right\rangle d u d r \\
=  \frac{1}{\epsilon}\int_0^t \int_{u-\epsilon}^{u} \left\langle \X_u, A^*\Gamma(r,\X_r) \right\rangle d r d u 
+  R(\epsilon,t)
\xrightarrow{\epsilon \to 0} \int_0^t \left\langle \X_u, A^*\Gamma(u,\X_u) \right\rangle d u.
\end{multline}
The validity of hypothesis (i) comes out setting  
$C_t = \int_0^t (\sigma_r Q^\frac{1}{2}) (\sigma_r Q^\frac{1}{2})^\ast dr.$
It holds because $\X$ is a $\bar \nu_0$-semimartingale, taking into account Proposition \ref{P67} 4. and item 1. of Proposition \ref{CovMart}.
 Expression \eqref{eq:first-Dinkin}  results now from \eqref{eq:Ito}.
The first integral of the right-hand side of \eqref{eq:Ito}
gives the second and third 
 integrals of  \eqref{eq:first-Dinkin}.
Those  are obtained
differentiating $\M$ and $\int_0^\cdot b_r dr$, using Remark \ref{RScalarCase} 2., 3. and Proposition \ref{PChainRule-2}.
The fourth integral comes from \eqref{eq:era38} choosing
$ \Gamma = DF$.
 Finally the last integral  in (\ref{eq:first-Dinkin}) comes from the 
third addendum of (\ref{eq:Ito}),
taking into account (\ref{EE3})  and Lemma \ref{LCovMart}
with $j = C$ and ${\dot g}(r) = D^2 F(r,\X_r)$.
\end{proof}

The  theorem below  operates  as a substitute of a non-smooth It\^o formula.
It is a stability of $\nu$-weak Dirichlet processes, which was the object of Theorem 4.2 of \cite{RusFab}.

\begin{thm} 
\label{th:prop6}
Let $H$ be a separable Hilbert space. 
Let $\nu_0$ be a dense subset of $H^\ast$. We set $\nu = \nu_0 \otimes \R$ 
and $\chi = \nu_0 \hat \otimes_\pi \nu_0$.
Let    $F \in C^{0,1}([0,T] \times H)$ such that $DF$ is continuous from 
$[0,T] \times H$ to $\nu_0$.
Let $\X = \M + \V$ be a $({\mathscr F}_t)-\nu$-weak Dirichlet process and we suppose that
$\X$ has a  $\chi$-quadratic variation. \\
Then $(F(t, \X_t)) $ is a real  $({\mathscr F}_t)$-weak Dirichlet process
with martingale part $M^u$
where
$$ M^F_t = F(0, \X_0) + \int_0^t \langle D F(r, \X_r), d\M_r \rangle_H.$$
\end{thm}

\section{Calculus with respect to window processes}

\label{Swindow}

Let $X$ be a real (continuous) process such that $[X]_t \equiv \psi(t)$
with $\psi(t) = \sigma^2 t $,
$\sigma \ge 0$. Let $B = C([-T,0])$.
Let $u: [0,T] \times B \rightarrow \R$ be of class $C^{0,1} ([0,T[ \times B)$. 
For $t \in [0,T[, \eta \in B$, we set 
$$ D^{\delta_0} u (t,\eta) = Du(t,\eta)({0}), \
D^\perp u(t,\eta) = Du(t,\eta) -  Du(t,\eta)({0}).$$
In this section, for $t \in [0,T]$, we will  denote again
\begin{equation} \label{D_t}
 D_t := \left \{ (x,y) \in [-t,0]^2 \vert x=y \right \}. 
\end{equation}

\subsection{The case of {\it vanilla} random variables}

\label{Svanilla}

A window diffusion $X$ is naturally 
related to an infinite dimensional Kolmogorov type equation.
But in fact this link remains valid when the process is $X$
 is 
a non-semimartingale with the same quadratic variation.
Let us concentrate on the case of a (non-necessarily semimartingale)
 process $X$ such that $[X]_t = \sigma^2t, t \in [0,T]$, for 
$\sigma \ge 0$.

In order to motivate the discussion, we start with the simple representation  
 a r.v. of the type $h= f(X_T)$
where $f: \R \rightarrow \R$ is continuous with polynomial growth.
We  suppose the existence of
 $v\in C^{1,2}([0,T[\times \R)\cap C^{0}([0,T]\times \R)$ such that 
$$
\left\{
\begin{array}{l}
\partial_{t}v(t,x)+\frac{\sigma^2}{2}\partial^2_{xx}v(t,x)=0\\
v(T,x)=f(x).
\end{array}
\right.
$$
Then
$$
h:=f(X_{T})=v(0,X_{0})+ \int_{0}^{T} 
\partial_{x} v(s,X_{r}) d^{-}X_{r},
$$
where previous integral is an improper forward integral.
That result appeared in \cite{schoklo},  \cite{z02}.
The proof can be easily formulated through Proposition \ref{PITO}. 
Later on,  generalizations were performed in the case of
 Asiatic options and other classes in \cite{crnsm, VaSoBen} 
and \cite{crnsm2}, which also considers r.v. of the type
   $h = f(X_{t_0}, \ldots, X_{t_N})$, for $ 0 = t_0 < \ldots  < t_N = T$.

The natural question concerns the validity of a similar
 formula when $h$ is path dependent.

Previous toy model can be revisited using infinite dimensional calculus  via regularization. 
\begin{prop}
We set again  $B=C([-T,0])$ and $\eta\in B$ and we define
$G:B \longrightarrow \R$, by $G(\eta):=f(\eta(0))$ 
and $u:[0,T]\times B \longrightarrow \R$, by $u(t,\eta):=v(t,\eta(0))$.
Then $u\in C^{1,2}\left( [0,T[\times B ;\R \right)\cap 
C^{0}\left( [0,T]\times B ;\R \right)$ and  it solves
\begin{equation} \label{Evanilla} 
\left \{
\begin{array}{ccc}
\partial_{t}u(t,\eta)&+&\frac{\sigma^2}{2}\langle D^{2}u\,(t,\eta)\; ,\;
 1_{D_t}\rangle=0, \ (t,\eta) \in [0,T| \times B,   \\
u(T,\eta)  &=& G(\eta), \ \eta \in B.  
\end{array}
\right.
\end{equation}
\end{prop}
\begin{proof}
The final condition is obviously verified since 
$u(T,\eta)=v(T,\eta(0))=f(\eta(0))=G(\eta)$ for all $\eta \in B$. Moreover 
$u$ is obviously of class $C^{1,2}([0,T[ \times B) \cap C^{0}([0,T] \times B)$
 and
 $ \partial_{t}u\,(t,\eta)=\partial_{t}v\,(t,\eta(0))$; also $ D u\,(t,\eta)
=\partial_{x}v\,(t,\eta(0))\;\delta_{0} $
and $ D^{2} u\,(t,\eta) =\partial_{x\, x}^{2}v\,(t,\eta(0))\; 
\delta_{0}\otimes \delta_{0}$.
Finally $ \partial_{t}u\,(t,\eta)+\frac{\sigma^2}{2}
D^{2}u\,(t,\eta)(D_t)=0 $.
\end{proof}

Suppose that $u:[0,T] \times B \rightarrow \R$
is of class $C^{0,1}([0,T[ \times B)$. 
A quantity which will play a role in the sequel is 
the deterministic forward integral 
 $\int_{]-t,0]} D_{dx}^{\perp} u(t,\eta) d^- \eta(x)$, 
see Section \ref{sdetint}.

Suppose that for a given $(t,\eta)$, $D_{dx}^{\perp} u(t,\eta)$ is absolutely
continuous, we denote by $x \mapsto D_x^{\mathrm ac} u(t,\eta)$  
 the corresponding derivative.
If moreover  
$x \mapsto D_x^{\mathrm ac} u(t,\eta)$ has bounded variation, then  
previous deterministic integral exists and, 
$$\int_{]-t,0]} D_{dx}^{\perp} u(t,\eta) d^- \eta(x)  = \eta(0) 
D_x^{ac} u(t,\eta)(\{0\}) - \int_{]-t,0]} \eta(x) D_{dx}^{ac} u(t,\eta(x)),$$
because of Remark \ref{BVI} 2.

In the toy model mentioned above, that integral  is clearly zero since
 $D^{\perp} u$ is identically zero.

\subsection{It\^o formulae for window processes}
\label{SitoWindow}

The It\^o formula stated  in 
Theorem \ref{TITOF} can be particularized for the case when
$\X = X(\cdot)$ is a window process (with $\tau = T$), associated with a finite quadratic variation process $X$.
We recall that $\X$ admits a $\chi_0$-quadratic variation, where $\chi_0$ is the
 Chi-space of signed measures on $[-T,0]^2$ introduced in item 6. before Proposition \ref{PChiCalculation}.
In particular Theorem \ref{TITOF} applies, so that  integral \eqref{FIFOF} exists and it it decomposes in the sum
\begin{equation} \label{ISum}
\int_0^t D^{\delta_0} F(r,\X_r) d^- X_r +   \int_{0}^{t} {}_{B^{\ast}} {\langle}  
D^\perp F (r,\X_r),  d^- \X_r\rangle_B,
\end{equation}
where $ D^{\delta_0} F$ and $ D^{\perp} F$ were defined at the beginning of  Section \ref{Swindow},
provided that at least one of the two addends exist.
\begin{rem} \label{R111}
\begin{enumerate}
\item The second term is the limit in probability of the expression
\begin{equation} \label{E112}
 \int_0^t dr   \int_{-r}^0 D^\perp_{dx} F(r, X_r) \frac{X_{r + x + \varepsilon} -
 X_{r+x}}{\varepsilon},
 \end{equation} 
when $\varepsilon$  goes to zero.
\item If $X$ is a semimartingale, then the first integral is the It\^o integral
$ \int_s^t D^{\delta_0} F(r, \X_r) dX_r$. Consequently the second one is forced to exist.
\item If $X$ is not necessarily a semimartingale, sufficient conditions for its 
existence can be provided. Suppose that  the deterministic quadratic variation
of almost all path of $X$ exists. In particular $[X]$ exists as an
 increasing real process.
%QUESTA IPOTESI PUO' FORSE ESSERE MIGLIORATA PASSANDO ALLE SUCCESIONI
 In this case a sufficient condition for the existence of the second integral 
in \eqref{ISum}, is the realization of
 following Condition related to $F$.
 We recall that the space $V_{2,\psi}$, 
 for a fixed  increasing continuous function
$\psi: [0,T] \rightarrow \R$ such that $\psi(0)=0$,
 was defined at Section \ref{sdetint}.
$B$ denotes here $C([-T,0])$.
\end{enumerate}
\end{rem}
\begin{dfn}\label{ConditionC}
A continuous function $u:[0,T] \times B \rightarrow \R$
of class $C^{0,1}([0,T[ \times B)$ is said to fulfill 
{\bf Condition} (C) (related to $\psi$) if the following holds.
\begin{enumerate}
\item For each $t \in [0,T], \eta \in V_{2,\psi}$, 
% the measure $D_{dx}^\perp u(t,\eta)$ is absolutely continuous and
the deterministic integral 
\begin{equation} \label{DetInt}
%\int_{[-t,0]}
 \int_{|-t,0]}   D_{x}^{\perp}  u(t,\eta) d^- \eta(x)
\end{equation}
 exists.
%where $ (D_{x}^{ac}  u(t,\eta))$ is the density towards Lebesgue measure.
\item For any $\varepsilon > 0, t \in [0,T], \eta \in B$, 
we denote
\begin{equation} \label{39bis}
 I(t, \eta,  \varepsilon) :=
%\int_{[-t,0]} 
\int_{-t}^0 D_{x}^{\perp}  u(t,\eta) \frac{\eta(x+\varepsilon) -
 \eta(x)}{\varepsilon} dx.
\end{equation}
We suppose the existence of $J: [0,T] \times V_{2,\psi} \rightarrow \R_+$ such
 that
$ \vert I(t, \eta,  \varepsilon) \vert \le  J(t,\eta), \forall \eta \in V_{2,\psi}$
and such that for each  compact $K$  of $B$ included in $ V_{2,\psi}$,
$ 
%\int_{-T}^0
\int_{0}^T  \sup_{\eta \in K} J(s,\eta) ds < \infty $.
\end{enumerate}
\end{dfn}
\begin{rem} \label{RNew}
\begin{enumerate}
\item
As far as last point is concerned we remark that relatively compact subsets 
of $B$ are
very tiny.
\item Sufficient conditions for the validity of Condition (C) will be given below.
Clearly this condition implies the existence of the second integral of 
\eqref{ISum}.\\
In fact the set $\{ X_s (\cdot), s \in [-T,0] \}$ is compact 
in $B$ and included in $V_{2,\psi}$.
If Condition (C) is verified then 
$$ \vert I(s, X_s(\cdot),  \varepsilon) \vert \le \sup_{\eta \in K}  J(s, \eta), 
$$
where $K = K(\omega) = \{ X_s (\cdot), s \in [-T,0] \}$. The result follows by
Lebesgue dominated convergence theorem.    
\item
\end{enumerate}

\end{rem}
A sufficient condition for the realization of Condition (C) is given below.
This will be a consequence  of integration by parts
and  It\^o chain rule  \eqref{ItoChain} in Proposition \ref{ITOFQV},
expressed in the context of 
deterministic calculus via regularization.

\begin{lem} \label{LTechnical}
Suppose the  existence of continuous maps
$F_i: [0,T] \times B \times [-T,0] \rightarrow \R$,
$G_i: [0,T] \times [-T,0] \times \R \rightarrow \R$, $1 \le i \le N$,
such that $D_{dx}^\perp u(t,\eta) $ is absolutely continuous and
$D_{dx}^{\perp} u(t,\eta) = D_{x}^{ac} u(t,\eta) dx
 = \sum_{i=1}^N F_i(t,\eta,x) G_i(t, x, \eta(x)) dx$, 
fulfilling the following properties for any
subset $K$  of $V_{2,\psi}$ such that 
\begin{equation} \label{FKK}
\sup_{\eta \in K} \Vert \eta \Vert_{2,\psi}  < \infty  \text{ and }
 K \text{ it is a compact subset of } 
B.
\end{equation}
For every $ 1 \le i \le N$, for any such $K$, we suppose the following.  
\begin{itemize}
\item For any $(t,\eta) \in [0,T] \times V_{2,\psi}$, 
  $F_i(t,\eta, \cdot)$ has bounded variation.
\item $(t,\eta) \mapsto \Vert F_i(t,\eta,\cdot) \Vert_{var}$
is   bounded on $[0,T] \times K$. 
\item $ G_i \in C^{0,1}([0,T] \times ([-T,0] \times  \R))$.
%for all $1 \le i \le N$. 
\end{itemize}
Then $u$ fulfills  Condition (C) with respect to $\psi$.
\end{lem}
\begin{proof}
By additivity we can reduce to the case $N = 1$ and we set $F = F_1$, $G = G_1$.
Let $K$ be a subset of $V_{2,\psi}$ such that \eqref{FKK} is fulfilled. \\
 Let $F: [0,T] \times V_{2,\psi} \times [-T,0] \rightarrow \R$
be measurable such that for every $t$  and $\eta$ we suppose that 
$x \mapsto F(t,\eta,x)$ 
has bounded variation  and $(t,\eta) \mapsto \Vert F_i(t,\eta,\cdot) \Vert_{var}$
is   bounded on $[0,T] \times K$.
 Let  $ G: [0,T] \times [-T,0] \times \R \rightarrow \R$ of class 
$C^{0,1}([0,T] \times ([-T,0] \times \R))$.
Let $\eta \in V_{2,\psi}$ and set
 $\tilde G: [0,T] \times [-T,0] \times \R \rightarrow \R$  the
 primitive defined 
by $\tilde G(t, x, y) = \int_{0}^y G(t,x, \tilde y) d \tilde y$. 
By formula \eqref{ItoChain} in Theorem \ref{ITOFQV}
and Remark \ref{BVI} 3.,  we obtain 
\begin{eqnarray} \label{FL2}
\int_{]-t,0]} 
% \int_{-t}^0 
F(t,\eta, x) G(t,x,\eta(x))d^- \eta(x) &=&  
\int_{]-t,0]}
%\int_{-t}^0
   F(t,\eta, x) d^-_x \tilde G(t,x,\eta(x))  
 -  \frac{1}{2} 
%\int_{-t}^0 
\int_{]-t,0]}
F(t,\eta, x) \partial_{\eta(x)} 
 G(t,x,\eta(x)) d[\eta](x) \nonumber\\
& & \\
&-& \int_{]-t,0]}  F(t,\eta, x) \partial_{x} \tilde G(t,x,\eta(x)) dx, 
 \nonumber 
\end{eqnarray}
provided that the first integral after the equality symbol is well-defined.
By
% Proposition \ref{IBP},
Remark \ref{BVI} 2., that integral equals
$$ F(t,\eta,0-) \tilde G(t,0,\eta(0)) - 
 %F(t,\eta,-t) \tilde G(t,-t,\eta(-t)) -
 \int_{]-t,0]} F(t,\eta, dx)  \tilde G(t,x,\eta(x));$$
consequently 
item (a) of Condition (C) is fulfilled.
\\
In the sequel of the proof, for $\eta \in B$, we denote by 
$R_K(t, \eta, \varepsilon)$
a quantity such that for every $0 < \varepsilon < 1$,
$\sup_{t \in [0,T], \eta \in K} \vert R(t, \eta, \varepsilon) \vert \le C(T, K)$,
where $C(T,K)$ only depend on $T$ and $K$.
We denote 
\[
K_0 := \bigcup_{\eta \in K} Im (\eta),
\]
which is clearly a compact subset of $\mathbb{R}$.
We need to control the quantity
\begin{equation}
\label{A11}
\int_{-t}^0 F(t,\eta,x) G(t,x,\eta(x)) \frac{\eta(x+\varepsilon) - 
\eta(x)}{\varepsilon} dx.
\end{equation}
We set again $\tilde G(t,x,y) = \int_{0}^y G(t,x,\tilde y) 
d \tilde y$ for $t\in [0,T]$ and $x\in [-T,0], y \in \R$, so that 
 $\partial_y \tilde G(t,x,y) = G(t,x,y)$. By Taylor
 expansion (\ref{A11}) equals
\[
I_1(t,\eta,\varepsilon) - I_2(t,\eta,\varepsilon) - I_3(t,\eta,\varepsilon)
 + R_K(t,\eta,\varepsilon),
\]
where
\[
I_1(t,\eta,\varepsilon) := \int_{-t}^0 \frac{dx}{\varepsilon} 
F(t,\eta, x) \left ( \tilde G(t,x+\varepsilon,\eta(x + \varepsilon) ) - \tilde G(t,x,\eta(x) ) \right ),
\]
\[
I_2(t,\eta,\varepsilon) := \int_{-t}^0 \frac{dx}{\varepsilon}
 F(t,\eta,x) \partial_x \tilde G(t,x + a \varepsilon,\eta(x) ),
\]
\[
I_3(t,\eta,\varepsilon) := \frac{1}{2} \int_{-t}^0 \frac{dx}{\varepsilon}\, 
F(t,\eta,x) \int_0^1 da \partial_y \tilde G
%(t,x,\eta(x)) + \left ( \eta(x+\varepsilon) - \eta(x) \right )^2.
(t,x+a \varepsilon,\eta(x) + a (\eta(x+\varepsilon) - \eta(x))) 
\left ( \eta(x+\varepsilon) - \eta(x) \right )^2.
\]
$I_1(t,\eta,\varepsilon)$ equals
\begin{multline}
\int_{-t}^0 \frac{dx}{\varepsilon} \left ( F(t,\eta,x) - 
F(t,\eta, x-\varepsilon) \right ) \tilde G(t,x,\eta(x) ) + R_K(t,\eta,\varepsilon)\\
= \int_{[-t,0]} \frac{F(t, \eta, dy)}{\varepsilon} \int_y^{y+\varepsilon} dx
 \, \tilde G(t,x,\eta(x)) + R_K(t,\eta,\varepsilon).
\end{multline}
Consequently, for $0 < \varepsilon < 1$ we have
\begin{equation}\label{E53bis}
|I_1(t,\eta,\varepsilon)| \leq \sup_{\begin{subarray}{l}
                                      x\in[-T,0]\\
                                      y \in K_0\\
                                      t\in[0,T]
                                     \end{subarray} } 
|\tilde G(t,x,y)| \sup_{\eta\in K, t \in [0,T]} \| F(t,\eta,\cdot) \|_{var}
+  \sup_{\begin{subarray}{l}
                                      \eta \in K\\
                                      0 < \varepsilon < 1  \\
                                      t\in[0,T]
                                     \end{subarray} } 
% \sup_{\eta\in K, t \in [0,T],   }
  \vert R_K(t,\eta,\varepsilon) \vert  =: C_1;
\end{equation}
$I_2(t,\eta,\varepsilon)$ can be handled in similar
(but easier) way to $I_1$. There is a constant $C_2$ such that
\[
\sup_{\begin{subarray}{l}
              x\in[-T,0]\\
              \eta \in K\\
              t\in[0,T]
      \end{subarray} }
|I_2(t,\eta,\varepsilon)| \leq C_2,
\]
getting a similar estimate as in \eqref{E53bis},
but replacing $\tilde G$ with $\partial_x \tilde G$.
Concerning $I_3(t,\eta,\varepsilon)$, for $0 < \varepsilon < 1$, we have
\[
|I_3(t,\eta,\varepsilon)| \leq 
 \sup_{\begin{subarray}{l}
                                      x\in[-T,0]\\
                                      y \in K_0\\
                                      t\in[0,T]
                                     \end{subarray} } 
\left |\partial_y \tilde G(t,x,y) \right |
\sup_{\begin{subarray}{l}
              x\in[-T,0]\\
              \eta \in K
      \end{subarray} }
\vert F(t,\eta,x) \vert
 \, \left (  \sup_{\eta\in K} \| \eta \|_{2,\psi}\right ),
\]
%for some compact set $K_0$ depending on $K$. 
which is bounded because of  \eqref{FKK}.
Finally item (b) of condition (C) is also fulfilled.
\end{proof}

%Let $a < b$ two real numbers and we suppose $0 \in [a,b]$, without restriction of generality. 

\subsection{An infinite dimensional PDE}
\label{SIDPDE}

In this subsection again $B$ will stand for
$C([-T,0])$.
We are interested here in a class of functionals 
 $ G: B\longrightarrow \R$  such that the r.v. 
$
{ h:=G(X_{T}(\cdot))}    %\quad \mbox{Borel functional}
$
admits a representation 
\begin{equation} \label{EGenClark}
{ h=G_{0}+\int_{0}^{T}\xi_{s}d^{-}X_{s}, }
\end{equation}
where $G_0$ is a real number and $\xi$ is adapted
with respect to the canonical filtration $(\mathscr F_t)$ of $X$.
If $X$ is a classical Wiener process, and $h$ belongs to some suitable
Malliavin type Sobolev space, then $G_0 = E(h)$ and  Clark-Ocone formula 
says that $\xi$ in \eqref{EGenClark} is given by
 $\xi_t = \E(D^m_t h \vert {\mathscr F}_t), t \in [0,T]$.
% in \eqref{EGenClark}. 

In this section we want to show that the replication of a random variable
 $h = G(X(\cdot))$, is robust with respect to the quadratic variation
of $X$, for a large class of $G$; the
fact that the underlying process is distributed according to
Wiener measure is not so relevant.
We are indeed interested in a representation 
\eqref{EGenClark},
which formulates $G_0$ and $\xi$ through two functionals of $X$,
which do not depend on the specific model of $X$ such that
$[X]_t \equiv \sigma^2 t, t \in [0,T]$.

The methodology 
for expressing a {\it Clark-Ocone type formula} for finite quadratic variation
processes 
 consists in two steps.
\begin{enumerate}
\item We need to choose a functional $u: [0,T]\times B \longrightarrow \R$ which 
 solves the infinite dimensional PDE \eqref{eq SYST} with final condition $G$.
\item Using an It\^o type formula we establish a representation form
 \eqref{EGenClark}.
\end{enumerate}
The proposition below represents the second step of the
procedure.

Below $\psi$ will stand for $\psi(t) \equiv \sigma^2 t$,
for some $\sigma \ge 0$.
\begin{prop}  \label{TReprClark}
Let $X$ a process such that a.s. the limit $[X,X]$ in Definition
 \ref{D33} holds a.s. and gives $\psi$. 
Let    $u:[0,T]\times B \longrightarrow \R$ be a function of class 
$C^{1,2}\left( [0,T[\times B \right)\cap C^{0}\left( [0,T]\times
  B\right)$. 
For $(t,\eta) \in [0,T] \times B$, we decompose 
$ D u(t,\eta) =  D^{\delta_0} u (t,\eta) \delta_0 + D_{dx}^{\perp} u (t,\eta) $.
We symbolize again through $\chi_0$  the Chi-space 
constituted by specific Borel signed measures on $[-T,0]^2$ 
% $\shm([-T,0]^2)$ 
defined in item 6. in Section 
\ref{window}. 
We suppose the following.
\begin{enumerate}
\item $u$ fulfills Condition (C) 
and we denote 
\begin{equation} \label{F73}
 I(u)(t, \eta) := \int_{]-t,0]}  D^{\perp}_{dx} u (t,\eta) d^-\eta(x), \ (t,\eta)
\in [0,T] \times V_{2,\psi}.  
\end{equation}
 \item For all $t \in [0,T], \eta \in  V_{2,\psi}$, 
$D^2u(t,\eta) \in \chi_0$ and the map 
$(t,\eta) \mapsto D^2u(t,\eta)  $ is continuous with respect to the topologies of 
$[0,T] \times B$ and $\chi_0$. 
\item $u$ solves the
 solving the  infinite dimensional PDE
\begin{equation} \label{eq SYST}
\left\{
\begin{array}{l}
\partial_{t}u(t,\eta)+ I(u)(t,\eta) +\frac{\sigma^2}{2} \langle D^{2}u\,
(t,\eta)\; ,\; 1_{D_t}\rangle=0, \ 
 (t,\eta)
\in [0,T] \times V_{2,\psi},       \\
\\
u(T,\cdot)  = G.
\end{array}
\right.
\end{equation}
%where 
%$ D_t := \left \{ (x,y) \in [-t,0]^2 \vert x=y \right \}. $
\end{enumerate}
Then representation \eqref{EGenClark} holds 
with
$G_{0}= u(0,X_{0}(\cdot))$
%\item
and $\xi_{s}=D^{\delta_{0}}u(s,X_{s}(\cdot))$.
%\end{itemize}
\end{prop}
\begin{rem} \label{RCor}
 The condition on $X$ implies that 
$X$ is a finite quadratic variation process
and $[X]_t = \sigma^2 t$. This hypothesis is a  bit stronger but it
is fulfilled in almost the known models where $[X] = \psi$. 
A typical $X$ with this a.s. property is
 the sum of a Wiener process and a H\"older process $V$
with respect to an index $\gamma > \frac{1}{2}$.
\end{rem}
\begin{proof} 
The proof the proposition is a consequence of Theorem \ref{TITOF} and of the 
considerations following the statement of Condition (C).
In particular we remark that for all $t \in [0,T]$, a.s.  we have
\begin{equation} \label{EC}
I(u)(t,X_t(\cdot)) =  \int_{0}^{t} {}_{B^{\ast}} {\langle}  
D^\perp F (r,\X_r),  d^- \X_r\rangle_B. 
\end{equation}
\end{proof}

Coming back to the two steps mentioned at the beginning of 
Section \ref{SIDPDE},
Theorem 9.41 and Theorem 9.53 of  \cite{DGR}
 give some sufficient conditions to solve \eqref{eq SYST}.
This constitutes step 1. This can be done for instance 
in the two following cases.
\begin{enumerate}
\item $G$ has a smooth Fr\'echet dependence on $L^{2}([-T,0])$.
\item $h:=G(X_{T}(\cdot))=f \left(\int_{0}^{T}\varphi_{1}(s)d^{-}X_{s},
\ldots,\int_{0}^{T}\varphi_{n}(s)d^{-}X_{s}
 \right)$, 
\begin{itemize}
\item $f:\mathbb{R}^{n}\rightarrow \mathbb{R}$ continuous
  with linear growth 
\item $\varphi_{i}\in C^{2}([0,T];\mathbb{R}), 1 \le i \le N$.
\end{itemize}
\end{enumerate}
\begin{rem} \label{R766}
Suppose that $X=W$.
There are cases where the methodology developed here 
is operational and the classical Clark-Ocone formula
 does not apply.
For instance in some cases $h$ may be allowed even not to belong to $L^{1}(\Omega)$
and a fortiori   $h\notin \D^{1,2}$  or $h\notin L^{2}(\Omega)$, see
for instance Proposition 9.55 in \cite{DGR}.
\end{rem}

\begin{rem} \label{R777}
\begin{enumerate}
\item Remark that our representation theorems also holds 
when $\sigma = 0$.
\item In a work in preparation, the authors extend the present theory 
to the case when $X$ is replaced with the window of a generic diffusion process.
\item The present  approach was developed at the same time and independently than  functional It\^o's calculus
 of B. Dupire,  R. Cont, D. Fourni\'e, see e.g. \cite{dupire,  ContFourJFA}.
\end{enumerate}
\end{rem}

\section{Applications to study of Kolmogorov equations}
\label{sec:Kolmogorov}

In this section we illustrate how to use the tools of stochastic calculus via regularization in the study of solutions of forward Kolmogorov equations 
(i.e. Fokker-Planck equations) related to an evolution 
problem in infinite dimensions, for instance a stochastic PDE.
 Kolmogorov equations in infinite dimension constitute a classical field 
of study, they appear for example in quantum field theory and in 
stochastic reaction-diffusion.
We do not have here the ambition of summarizing the existing literature but only to describe how the development of the theory we have
 described in previous sections can help to treat some cases that are not covered by the existing literature.

We are interested in studying a class of  Kolmogorov equations associated to an evolution equation of the form \eqref{FEDPS} using the strong solution approach. 
In other words we will define the solution of the Kolmogorov equation 
using  approximating sequences, see Definition \ref{def:strong-sol}. The main results of the section are the following.
\begin{itemize}
 \item[(i)] We provide, first of all, in Theorem \ref{th:decompo-sol-Kolmogorov}, a probabilistic representation of strong solutions $(t,\eta) \mapsto 
v(t,\eta)$ of the Kolmogorov
 equation decomposing it into two stochastic terms: the evaluation of the initial datum of the Kolmogorov equation along the trajectory of a reversed
 evolution equation and a stochastic integral term depending on the first derivative of $v$. 
 \item[(ii)] In Proposition \ref{cor:verification} we show  that 
%%%%%%% EQUIVALENCE? GIORGIO PERCHE EQUIVALENZA. IN QUALE CLASSE? $C^{0,1}$?
 a strong solution  of the Kolmogorov equation is also a mild solution.
 The definition of mild solution will be recalled in  (\ref{eq:def-valuefunction}).
 As a corollary we get the uniqueness of the strong solution.
\end{itemize}
With respect to similar contributions in this sense (see e.g. \cite{CerraiGozzi95, Gozzi97, DaPratoZabczyk02}) we are able to prove the uniqueness of the strong solution in cases in which the stochastic evolution equation connected to the Kolmogorov equation is not homogeneous and in which the regularity of 
the solution $(t,\eta) \mapsto v(t,\eta)$ is only requested to be $C^{0,1}$ with $D v(t,\eta) \in C([0,T]\times H; D(A^*)),$ 
 where $A$ is
 the generator of
 the $C_0$-semigroup 
appearing in the infinite dimensional stochastic evolution equation, see e.g. \eqref{FEDPS},  related to the Kolmogorov equation. 
More details about comparison  with existing results are given in Example \ref{ese1} and Remark \ref{ese2}.

\subsection{The setting}

\label{SSetting}
Let $H$ be a separable Hilbert space  and  $A$ be the generator of a $C_0$-semigroup on $H$, see Section \ref{Conv}.
We denote again with $D(A)$ and $D(A^*)$ respectively the domains of $A$ and $A^*$ endowed with the graph norm.
Let fix again $T > 0$.

Let us again consider  $\left( \Omega,\shf,\mathbb{P}\right)$  
a complete probability space and $\left ( \mathscr{F}_t \right )_{t\geq 0}$ 
a filtration  
on it
% $\left (\Omega , \mathscr{F}, \mathbb{P} \right )$ 
satisfying the usual conditions.
 Assume that $U, U_0 $ are separable Hilbert spaces, $Q \in \mathcal{L}(U)$ is a positive, injective and self-adjoint operator 
as in Section \ref{SGPF} or \ref{Conv} and  define $U_0:=Q^{1/2} (U)$.
% endowed with the scalar product $\left\langle a,b \right\rangle_{U_0} := \left\langle Q^{-1/2}a,Q^{-1/2}b \right\rangle$. 
Let again $\W^Q=\{\W^Q_t:0\leq t < +\infty\}$ be an $U$-valued $(\mathscr{F}_t)$-$Q$-Wiener process, with $\W^Q_0=0$, $\mathbb{P}$ a.s.

We  consider two functions $b$ and $\sigma$ as follows.

\begin{Hypothesis}
\label{hp:onbandsigma}
$b\colon [0,T]\times  H \to H$ is a continuous function and satisfies, for some $C>0$,
\[
\begin{array}{l}
|b(t,\eta) - b(t, \gamma)| \leq C |\eta- \gamma|, \\[3pt]
|b(t,\eta)| \leq C (1+|\eta|),
\end{array}
\]
for all $\eta, \gamma \in H$, $t\in [0,T]$. $\sigma\colon [0,T]\times H \to \mathcal{L}_2(U_0; H)$ is continuous and, for some $C>0$,
\[
\begin{array}{l}
\|\sigma(t,\eta) - \sigma(t,y)\|_{\mathcal{L}_2(U_0; H)} \leq C 
|\eta- \gamma|, \\[3pt]
\|\sigma(t,\eta)\|_{\mathcal{L}_2(U_0; H)} \leq C (1+|\eta|),
\end{array}
\]
for all $\eta, \gamma \in H$, $s\in [0,T]$.% (being $\| \cdot \|_{\mathcal{L}_2(U_0, H)}$ the Hilbert-Schmidt norm).
\end{Hypothesis}

\medskip

\begin{rem}
\label{rm:sigmaQ}
Observe that, thanks to the definition of norm on $U_0$, 
the hypothesis $\|\sigma(t,\eta)\|_{\mathcal{L}_2(U_0; H)} \leq C (1+|\eta|)$ 
implies
\[
\|\sigma(t,\eta) Q^{1/2}\|_{\mathcal{L}_2(U; H)} \leq C (1+|\eta|), \ 
(t,\eta) \in [0,T] \times H
\]
and then
\[
\left \| \left ( \sigma (t,\eta) Q^{1/2} \right )
 \left ( \sigma (t,\eta) Q^{1/2} \right )^\ast \right \|_{\shl_1(H)}
 \leq C^2 (1+|\eta|)^2, 
(t,\eta) \in [0,T] \times H.
\]
\end{rem} 

\medskip
For  $\eta \in H$, we consider the equation
\begin{equation}
\label{eq:state}
\left \{
\begin{array}{l}
d \X_t = \left ( A\X_t+ b(t,\X_t) \right ) d t + \sigma(t,\X_t) d \W^Q_t,\\[5pt]
\X_0=\eta.
\end{array}
\right.
\end{equation}

The solution of (\ref{eq:state}) is understood in the mild sense, so an $H$-valued 
 predictable continuous process $\X$ is said to be a {\bf mild  solution} of (\ref{eq:state}) if
%, for all $T\in (0, +\infty)$,
%is said to be a solution of (\ref{eq:state}) 
\[
\mathbb{P} \left ( \int_0^T (|\X_r| + | b(r,\X_r)| + \| \sigma(r,\X_r)\|_{\mathcal{L}_2(U_0; H)}^2) d r <+\infty  \right ) = 1
\]
and
\begin{equation}
\label{eq:state-mild}
\X_t =  e^{tA}\eta + \int_0^t e^{(t-r)A} b(r,\X_r) d r
 + \int_0^t e^{(t-r)A} \sigma(r,\X_r) d \W^Q_r
\end{equation}
 $\mathbb{P}$-a.s. for every  $ t \in [0,T]$,

\bigskip
 
Thanks to Hypothesis \ref{hp:onbandsigma}, standard results about stochastic infinite dimensional evolution equation, see e.g. Theorem 3.3 of 
\cite{GawareckiMandrekar10}, ensure that there exists a unique solution
 $\X$ of (\ref{eq:state}), which  admits a continuous modification. 
So for us, the solution $\X$ can always  be considered as a continuous  process.

\subsection{The Kolmogorov equation}

Let 
%$l\colon [0,T] \times H \times \Lambda \to \mathbb{R}$  be a measurable function and 
$g\colon H\to\mathbb{R}$ be a continuous and bounded function.
We introduce now the following non-homogeneous Kolmogorov equation.
\begin{equation}
\label{eq:Kolmogorov}
\left\{
\begin{array}{l}
- \partial_t v + \left \langle  A^* D v, \eta \right\rangle + \frac{1}{2} 
{\mathrm Tr} \left [ \, \sigma(t,\eta) \sigma^*(t,\eta) D^2 v \right ] + %\\[3pt]
%\qquad\qquad  + \inf_{a\in \Lambda }
  \left \langle D v, b(t,\eta) \right\rangle %+ l(s,\eta) 
  =0 \qquad \text{$(t,\eta) \in [0,T] \times H$},\\ [8pt]
v(0,\eta)=g(\eta),\qquad\text{$\eta \in H$}.
\end{array}
\right.
\end{equation}
In the above equation, given $(t,\eta) \in [0,T] \times H$, given $v:[0,T] \times H \rightarrow H$,   $Dv(t,\eta)$
%$D v$
 (resp.
% $\partial^2_{xx} v$]
$D^2 v(t,\eta)$)  is the  Fr\'echet (resp. second Fr\'echet)   derivative of $v$ w.r.t. to the second variable  $\eta$; it is identified with elements of $H$
 (resp. with a symmetric bounded operator on $H$, taking into account the identification \eqref{eq:expressionLB}).
 $\partial_t v$ is the derivative w.r.t. the time variable.

We recall that the spaces $C([0,T] \times H), C(H), 
C([0,T] \times H; D(A^\ast))$ are Fr\'echet type spaces
if equipped with the topology defined by the seminorms
\eqref{NormC}.
We denote by $\mathscr{L}_0$ the operator on $C([0,T]\times H)$ defined as
\begin{equation}
\label{eq:def-A0}
\left \{
\begin{array}{l}
D(\mathscr{L}_0):= \left \{ \varphi \in C^{1,2}([0,T]\times H) \; : \; D \varphi \in C( [0,T] \times H ; D(A^*)) \right \} \\[8pt]
\mathscr{L}_0 (\varphi)(t,\eta) := - \partial_t \varphi (t,\eta) + \left \langle A^* D \varphi (t,\eta), \eta \right\rangle + \frac{1}{2}
 {\mathrm Tr} \left [  \sigma(t,\eta) \sigma^*(t,\eta) D^2 \varphi(t,\eta) \right ].
\end{array}
\right .
\end{equation}

Using this notation, (\ref{eq:Kolmogorov}) can be rewritten as
\begin{equation}
\label{eq:Kolmogorov-with-F}
\left\{
\begin{array}{l}
%\partial_s v =  
\mathscr{L}_0(v(t,\cdot)) + %F(s,x,\partial_xv)
\left \langle D v, b(t,\eta)\right\rangle =0,  \qquad \text{$(t,\eta) \in [0,T] \times H$},\\[6pt]
v(0,\eta)=g(\eta),\qquad \text{$\eta \in  H$}.
\end{array}
\right.
\end{equation}

\subsection{Mild, strict and strong solutions}

We recall here three different definitions of solution of the Kolmogorov 
equation, see e.g. \cite{DaPratoZabczyk02} for more details.
Assume that Hypothesis \ref{hp:onbandsigma} is verified. Fix $s\in ]0,T]$.
By the same arguments as those at the end of Section \ref{SSetting},
 the equation below has a unique mild
 solution $\Y^s$ on $[0,s]$:
\begin{equation}
\label{eq:state-Y}
\left \{
\begin{array}{l}
d \Y^s_t = \left ( A\Y_t+ b(s-t,\Y^s_t) \right ) d t + \sigma(s-t,\Y^s_t) d \W^Q_t, \qquad \text{$t \in [0,s]$}, \\[5pt]
\Y^s_0=\eta.
\end{array}
\right.
\end{equation}
 We will be in fact mainly interested in its value at point $s$.

\begin{dfn} \label{D23}
[Mild solution of the Kolmogorov equation].

We call {\bf mild solution} of the Kolmogorov equation (\ref{eq:Kolmogorov}) the function
\begin{equation}
\label{eq:def-valuefunction}
V(s,\eta) : = %P_t(g) (x) %+  \int_{0}^{t} P_{t-r} \, l(r, \cdot) (x) d r 
 \mathbb{E}\bigg[ 
%\int_{0}^{t} l(r, \X(t-r;x)) d r + 
g (\Y^s_s) \bigg],
\end{equation}
where $\Y^s$ is the solution of (\ref{eq:state-Y}).
\end{dfn}

\begin{rem} \label{R84}
Whenever $b$ and $\sigma$ does not depend explicitly on  time we have  $\Y^s_s=\X_s$, where $\X$ is the solution of (\ref{eq:state}),
 so the definition given
 above reduces to the mild solution given in \cite{DaPratoZabczyk02} 
Section 6.5 page 122.
 In this case the mild solution can be expressed in terms of the transition semigroup
 $(P_t, t>0)$ corresponding to (\ref{eq:state}). More precisely one has
$V(t,\eta) : = P_t(g) (\eta),$
where, for any $t \in ]0,T]$, and for any bounded, measurable function $\phi\colon H \to \mathbb{R}$, $P_t$ is characterized as 
\begin{equation}
\label{eq:defPt}
(P_t \; \phi)   (\eta)= \mathbb{E} [\phi  (\X_t)].
\end{equation}
\end{rem}

We recall,  in a  slightly more general situation, the notion of strict and strong solutions.
Let consider $h\in C([0,T]\times H)$, $g\in C(H)$ and the  Cauchy problem
\begin{equation}
\label{eq:simil-Kolmogorov-con-h}
\left \{
\begin{array}{l}
%\partial_s v =  
(\mathscr{L}_0(v) + h)(t,\eta)=0, \qquad \text{$(t,\eta) \in [0,T] \times H$},\\[5pt] 
v(0,\eta) = g(\eta),\qquad \text{$\eta \in H$}.
\end{array}
\right .
\end{equation}
Moreover, for any $s\in (0,T]$,  we consider  the following Kolmogorov equation with final datum:
\begin{equation}
\label{eq:simil-Kolmogorov-con-inverse-pre}
\left \{
\begin{array}{l}
\partial_t u(t,\eta) + \left \langle  A^* D u(t,\eta), x \right\rangle + \frac{1}{2} {\mathrm Tr} \left [ \, \sigma(t,x) \sigma^*(t,x) D^2 u \right ] + %\\[3pt]
 h(t,\eta)=0, \qquad \qquad \text{$(t,\eta) \in [0,s[ \times H$}, \\[5pt] 
u(s,\eta) = g(\eta), \qquad \text{$\eta \in H$}.
\end{array}
\right .
\end{equation}
Introducing the new notation
\begin{equation}
\label{eq:def-L0-s}
\left \{
\begin{array}{l}
D(\mathscr{L}^s_0):= \left \{ \varphi \in C^{1,2}([0,s]\times H) \; : \; D \varphi \in C( [0,s]\times H ; D(A^*)) \right \}, \\[8pt]
\mathscr{L}^s_0 (\varphi)(t,\eta) :=  \partial_t \varphi (t,\eta) + \left \langle A^* D \varphi (t,\eta), \eta \right\rangle + \frac{1}{2} {\mathrm Tr}
 \left [  \sigma(t,\eta) \sigma^*(t,\eta) D^2 \varphi(t,\eta) \right ],
\end{array}
\right .
\end{equation}
  equation 
 \eqref{eq:simil-Kolmogorov-con-inverse-pre}
can be rewritten as
\begin{equation}
\label{eq:simil-Kolmogorov-con-inverse}
\left\{
\begin{array}{l}
%\partial_s v =  
(\mathscr{L}^s_0(u) + 
h)(t,\eta) = 0, \qquad \text{$(t,\eta) \in [0,s[ \times H$},\\[6pt]
u(s,\eta)=g(\eta),\qquad \text{$\eta \in H$}.
\end{array}
\right.
\end{equation}

\medskip

\begin{rem} \label{R85}
Observe that the signs in front of $\partial_t$ are opposite in (\ref{eq:Kolmogorov}) and 
(\ref{eq:simil-Kolmogorov-con-inverse-pre}).
\end{rem}

\bigskip

\begin{dfn}\label{D86} [Strict solution of the Kolmogorov equation].

Consider $h\in C([0,T]\times H)$ and $g\in C(H)$. We say that $v \in C^{1,2}([0,T]\times H)$ (resp. $u \in C([0,s]\times H)$) is a {\bf 
strict solution} of 
(\ref{eq:simil-Kolmogorov-con-h}) (resp. of (\ref{eq:simil-Kolmogorov-con-inverse})) if $v\in D(\mathscr{L}_0)$ (resp. if $u\in D(\mathscr{L}^s_0)$) 
 and (\ref{eq:simil-Kolmogorov-con-h}) (resp. (\ref{eq:simil-Kolmogorov-con-inverse})) is satisfied.
\end{dfn}

\smallskip
\begin{dfn}[Strong solution of the Kolmogorov equation].
\label{def:strong-sol}

Let $h\in C([0,T]\times H)$ and $g\in C(H)$. We  say that 
$v \in C^{0,1}([0,T]\times H)$ with $D v \in C([0, T]\times H; D(A^*))$
 (resp. $u \in C^{0,1}([0,s]\times H)$ with $D u \in C([0,s]\times H; D(A^*))$)  is a {\bf strong solution} of
 (\ref{eq:simil-Kolmogorov-con-h}) (resp. of (\ref{eq:simil-Kolmogorov-con-inverse})) if there exist three sequences
 $\{v_n \} \subseteq D(\mathscr{L}_0)$ (resp. $\{u_n \} \subseteq D(\mathscr{L}^s_0)$), $\{h_n\} \subseteq C([0,T]\times H)$ 
(resp. $C([0,s]\times H)$) and 
$\{ g_n \} \subseteq C(H)$ fulfilling the following.
\begin{enumerate}
\item[(i)] 
For any $n\in\mathbb{N}$, $v_n$ (resp. $u_n$) is a strict solution of the 
problem
\begin{equation}
\label{eq:approximating}
\left \{
\begin{array}{l}
 \mathscr{L}_0(v_n)(t,\eta)) + h_n(t,\eta)=0, \qquad \text{$(t,\eta) \in [0,T] \times H$},\\[5pt]
v_n(0,\eta) = g_n(\eta)\qquad \text{$\eta \in H$}.
\end{array}
\right .
\end{equation}
\begin{equation}
\label{eq:approximating-inverse}
\left (\text{resp. of}
\left \{
\begin{array}{l}
 (\mathscr{L}^s_0(u_n) + h_n)(t,\eta)=0,\qquad \text{$(t,\eta) \in [0,s[ 
\times H$},\\[5pt]
u_n(s,\eta) = g_n(\eta),\qquad \text{$\eta \in H$}.
\end{array}
\right. \right)
\end{equation}
\item[(ii)] The following convergences hold:
\[
\begin{array}{ll}
\left \{
\begin{array}{ll}
v_n\to v & \text{in} \;\; C([0,T]\times H),\\
h_n\to h & \text{in} \;\; C([0,T]\times H),\\
g_n\to g & \text{in} \;\; C(H),
\end{array}
\right .

\qquad\qquad
\left ( 
resp. \; 
\left \{
\begin{array}{ll}
u_n\to u & \text{in} \;\; C([0,s]\times H),\\
h_n\to h & \text{in} \;\; C([0,s]\times H),\\
g_n\to g & \text{in} \;\; C(H).
\end{array}
\right .
\right )
\end{array}
\]
\end{enumerate}
\end{dfn}

\subsection{Decomposition for strong solutions of the Kolmogorov equation}

\begin{thm}
\label{th:decompo-sol-Kolmogorov}
Consider %$h\in C([0,T]\times H)$ and 
$g\in C(H)$.  Assume that Hypothesis \ref{hp:onbandsigma} is satisfied.
 Suppose that $v\in C^{0,1}([0,T]\times H)$ with $D v\in C(H; D(A^*))$ is a
 strong solution of (\ref{eq:simil-Kolmogorov-con-h}).
Then, given $s\in ]0,T]$  and $\eta\in H$, we have 
\begin{equation}
v(s, \eta) = g(\Y^s_s) - \int_0^s \left\langle D v(s-r, \Y^s_r), \sigma (s-r,\Y^s_r) d \W^Q_r \right\rangle,
\end{equation}
where $\Y^s$ is the solution of (\ref{eq:state-Y}).
\end{thm}

% \end{ese}
\begin{proof}%[Proof of Theorem \ref{th:decompo-sol-Kolmogorov}]
We  denote by $(v_n)$ the sequence of
smooth solutions of the approximating problems prescribed 
by Definition \ref{def:strong-sol}, which converges to $v$. We fix $s>0$ and we observe that $t\mapsto u(t,\eta):= v(s-t,\eta)$ is a strong solution of
\begin{equation}
\label{eq:Kolmogorov-with-F-inversa}
\left\{
\begin{array}{l}
\partial_t u+  \left \langle A^* D u, \eta \right\rangle + 
\frac{1}{2} {\mathrm Tr} \left [ \sigma(s-t,\eta) \sigma^*(s-t,\eta) D^2 u \right ] + \left \langle D u,
 b(s-t,\eta)\right\rangle = 0,\\[6pt]
u(s,\eta)=g(\eta),
\end{array}
\right.
\end{equation}
in the sense of Definition \ref{def:strong-sol} (in the case of (\ref{eq:approximating-inverse})) if we use, as a approximating sequence, $u_n(t,\eta):=v_n(s-t,\eta)$.
Thanks to Proposition \ref{lm:Ito}, every $u_n$ verifies, for $t\in [0,s]$,
\begin{multline}
\label{eq:u-Ito-pre}
u_n(t,\Y^s_t) =  u_n(0,\eta)  + \int_0^t \partial_r {u_n}(r,\Y^s_r) d r\\  +  \int_0^t \left\langle A^* D u_n(r,\Y^s_r), \Y^s_r \right\rangle d r
 +  \int_0^t \left\langle D u_n(r,\Y^s_r), b(s-r, \Y^s_r) \right\rangle d r\\
+ \frac{1}{2}  \int_0^t \text{Tr} \left [\left ( \sigma (s-r, \Y^s_r) {Q}^{1/2} \right )
\left ( \sigma (s-r, \Y^s_r) Q^{1/2}  \right)^* D^2 u_n(s-r,\Y^s_r) \right ] d r\\
+ \int_0^t \left\langle D u_n(r,\Y^s_r), \sigma(s-r, \Y^s_r) d \W^Q_r 
\right\rangle. \qquad \mathbb{P}-a.s.
\end{multline}
Since $u_n$ is a strict solution of (\ref{eq:Kolmogorov-with-F-inversa}) the expression
above gives, for $t\in [0,s]$,
\begin{equation}
u_n(t,\Y^s_t) =  u_n(0,\eta)  + 
%\int_0^s {h_n}(r,\X(r)) d r\\ +  \int_0^s \left\langle D u_n(r,\X(r)), b_i(r, \X(r)) \right\rangle d r + 
\int_0^t \left\langle D u_n(r,\Y^s_r), \sigma(s-r, \Y^s_r) d W_r\right\rangle.
\end{equation}
Define, for $t\in [0,s]$,
\begin{equation}
M^n_t := u_n(t,\Y^s_t) -  u_n(0,\eta).
\end{equation}
$(M^n)_{n\in\mathbb{N}}$ is a sequence of real local martingales 
(vanishing at zero). Since, thanks to Theorem 7.4 of \cite{DaPratoZabczyk92} 
one has
\[
\mathbb{E} \sup_{t\in [0,s]} \left ( 1 + \left | \Y_t^s \right |^N \right )
 < +\infty \qquad \text{for \ any} \ N \geq 1,
\]
$M^n$ converges ucp, thanks to the definition of strong solution, to
\begin{equation} \label{E72}
M_t := u(t,\Y^s_t) -  u(0,\eta).
%  - \int_0^s {h}(r,\X(r)) d r - - \int_0^s \left\langle D u(r,\X(r)), b_i(r, \X(r)) \right\rangle d r.
\end{equation} 
Since the space of real continuous local martingales equipped with
  the ucp topology is closed (see e.g. Proposition 4.4 of \cite{rg2}) then 
%Thanks to Proposition \ref{pr:FRnew} 
$M$ is a continuous local martingale.
\medskip

Now set $ \nu_0 = D(A^*)$,  
$ \chi =  \nu_0 \hat \otimes_\pi  \nu_0$ and we show how the theory developed 
in the previous sections can help us here.
Proposition \ref{PConv} 2. 
 ensures that $\Y^s$  is a $\bar \nu_0$-semimartingale with $\bar \nu_0$
being the dual of $D(A^*)$. By Proposition \ref{P67} 3.,
%$\bar \nu_0$
it is a $\nu_0 \hat\otimes_\pi \R$-weak Dirichlet process 
%with finite 
%$\bar \chi$-quadratic variation 
with decomposition $\M + \A$
 where $\M$  is 
 the local martingale 
% (with respect to $\P$)
 defined by
 $\M_t = \eta + \int_0^t \sigma (s-r, \Y^s_r) d \W^Q_r$
and $\A$ is a $\nu_0 \hat\otimes_\pi \R$-martingale-orthogonal process.
Moreover $\X$ has a finite $\chi$-quadratic variation by Proposition \ref{PConv} item 4.

Theorem \ref{th:prop6}  and Proposition \ref{PChainRule-2} (ii)
 ensures that the process $u(\cdot, \Y^s_\cdot)$ is a real
 weak Dirichlet process whose local martingale part being equal to
\begin{equation} \label{E72bis}
N_t = u(0,\eta) + \int_0^t \left\langle D u(r,\X_r), \sigma(s-r, \Y^s_r) 
d \W^Q_r\right\rangle.
\end{equation}
Observe that \eqref{EChainRule11} is satisfied thanks to Remark \ref{rm:sigmaQ} 
and the continuity of $Dv$, $\X$ and $\Y^s$.

By item 1. of  Proposition \ref{rm:ex418}  the decomposition of a real
 weak Dirichlet process is unique so, identifying \eqref{E72} with
 \eqref{E72bis},
 for any $t\in [0,s]$,
we get
\begin{equation} \label{E72quater}
u(t, \Y^s_t) = u(0,\eta) + \int_0^t \left\langle 
D u (r,\Y^s_r), \sigma(s-r, \Y^s_r) d \W^Q_r\right\rangle.
\end{equation}
 Since $v(s,\eta) = u(0,\eta)$, for any $t\in [0,s]$, by \eqref{E72quater}
it yields
%\begin{multline}
\begin{equation} \label{E72quinque}
v(s,\eta) = u(0,\eta) = u(t,\Y^s_t) 
%- M_t = u(t,\Y^s_t) 
- \int_0^t \left\langle D u (r,\Y^s_r), \sigma(s-r, \Y^s_r) 
d \W^Q_r\right\rangle.
% \\
%= g(\Y^s_t) - \int_0^t \left\langle D v (s-r,\Y^s_r), \sigma(s-r, \Y^s_r) d \W^Q_r\right\rangle,
\end{equation}
%\end{multline}
In particular, for $t=s$, since $u(s, \cdot) = g$ by
\eqref{eq:Kolmogorov-with-F-inversa}, it follows 
\begin{eqnarray*}
v(s,\eta) &=& u(0,\eta) 
%= u(s,\Y^s_s) - M_s 
= g(\Y^s_s) - \int_0^s \left\langle D u (r,\Y^s_r), \sigma(s-r, \Y^s_r) d \W^Q_r\right\rangle \\
&=& g(\Y^s_s) - \int_0^s \left\langle D v (s-r,\Y^s_r), \sigma(s-r, \Y^s_r) d \W^Q_r\right\rangle,
\end{eqnarray*}
which concludes the proof.
\end{proof}

We are now able to establish  uniqueness of the solution of the
Kolmogorov equation.

\begin{prop}
\label{cor:verification}
Assume that Hypotheses \ref{hp:onbandsigma} are satisfied and that 
$g$ is a continuous function from $H$ to $\mathbb{R}$.
 Let $v\in C^{0,1}([0,T]\times H)$ with $D v\in C(H; D(A^*))$ be a
 strong solution of (\ref{eq:Kolmogorov}). 
 Let $v$ such that
 $D v$
 has at most polynomial growth in the $\eta$ variable. Then the following holds.
 \begin{itemize}
  \item[(i)] The expectation appearing in \eqref{eq:def-valuefunction} 
makes sense  and it is finite; consequently the function $V$ is well-defined.
 \item[(ii)] $v=V$ on $[0,T]\times H$.
 \end{itemize}

\end{prop}
\begin{proof}
Thanks to Theorem \ref{th:decompo-sol-Kolmogorov},
 we can write, for any $s\in(0,T]$,
 \begin{equation}
\label{ex66bis}
v(s, \eta) + \int_0^s \left\langle D v (s-r,\Y^s_r), \sigma(s-r, \Y^s_r) d \W^Q_r\right\rangle = g(\Y^s_s).
\end{equation}
Observe that, by Theorem 7.4 in \cite{DaPratoZabczyk92}, all the momenta 
of $\sup_{r \in [0,t]} \vert \Y^s_r \vert $ are finite. 
On the other hand $D v$ has polynomial growth,
then, recalling Remark \ref{rm:sigmaQ}, for $t\in [0,s]$, $$ \mathbb{E} 
\int_0^t \left\langle D v (s-r,\Y^s_r), \left ( \sigma (s-r,\Y^s_r) Q^{1/2} 
\right ) \left ( \sigma (s-r,\Y^s_r) Q^{1/2} \right )^* D v(s-r, \Y^s_r) 
\right\rangle d r$$
is less or equal to, for all $t\in [0,s]$,
$$ \mathbb{E} \int_0^t C \left ( 1 + |\Y^s_r|^N \right ) d r$$
for some constants $C$ and $N$ and then,  thanks again to Theorem 7.4 of 
\cite{DaPratoZabczyk92} is finite.

 Consequently, by Proposition \ref{PChainRule-2} (i)
%So, by  \cite{DaPratoZabczyk92} Section 3.4., the local martingale
\[
t\mapsto \int_0^t \left\langle D v (s-r,\Y^s_r), \sigma(s-r, \Y^s_r) d \W^Q_r\right\rangle, \qquad t\in [0,s], 
\]
is  a true martingale vanishing at $0$.  Consequently, for any $t\in [0,s]$,
 its expectation is zero. 
In  the left-hand side of  (\ref{ex66bis})
 we have a deterministic value and a random variable with zero-expectation,
 so the expectation of the right-hand side is well-defined and equals
  $v(s, \eta)$.
In particular we have
\[
v(s, \eta)  = \mathbb{E} \bigg [ g(\Y^s_s)) \bigg ] 
%+ \int_0^t l(r, \X(r)) d r 
\]
which  concludes the proof.
\end{proof}

\begin{ese}
\label{ese1}
Whenever $b$ and $\sigma$ do not depend directly on the time and then 
the Kolmogorov equation is homogeneous,  if  $\X$ is
 the solution of (\ref{eq:state}) and  $\Y^s$ the solution of (\ref{eq:state-Y}) we have 
\[
\X = \Y^s \qquad\text{on $[0,s]$},
\]
for any $s\in ]0,T]$. So in particular $v(s, \eta) = V(s,\eta) = P_s(g) (\eta)$ where $(P_t)$ is the transition semigroup associated to (\ref{eq:state}).
 In this case Proposition \ref{cor:verification} gives a result similar to that of Theorem 7.6.2 Chapter 7 of \cite{DaPratoZabczyk02}. 
In that case the authors do not use a strong solution approach. The two results have different hypotheses; in fact the one contained in \cite{DaPratoZabczyk02} 
requires that $v$ is in twice differentiable with locally uniformly continuous derivatives in the $\eta$ variable while our result require the $C^1$ regularity 
and that $D v(t,\eta) \in C([0,T]\times H; D(A^*))$.
\end{ese}

\begin{rem}
\label{ese2}
The technique we have presented here can easily be adapted to treat other cases. One is the case in which $b\equiv 0$ and the function $h$ appearing in (\ref{eq:simil-Kolmogorov-con-h}) is a generic continuous function.

In this case the uniqueness result can be formulated as follows:
 any strong solution with the regularity required by Proposition \ref{cor:verification} can be expressed as
\[
v(s, \eta)  = \mathbb{E} \bigg [ g(\Y^s_s) + \int_0^s h(s-r, \Y^s_r) d r \bigg ].
\]
Whenever $\sigma$ does not depend directly on the time the expression above can be rewritten as 
\[
v(s, \eta)  = \mathbb{E} \bigg [ g(\X_s) + \int_0^s h(s-r, \X_r) d r \bigg ].
\]
So the existence of $\mathbb{E} \left [ g(\X_s) \right ]$ implies the existence of $\mathbb{E} \left [\int_0^s h(s-r, \X_r) d r \right ]$ and vice-versa. When one of the two exists (e.g. if $g$ of $h$ are bounded or have polynomial 
growing) we can write the latter expression as
\[
 P_s(g) (\eta) + \int_0^s P_r( h(s-r, \cdot)) (\eta) dr.
\]
Then $v$ is the mild solution used for example (in the particular case $\sigma$ being the identity) in \cite{Gozzi97}. In that paper, the author uses a strong solution approach, introducing a series of functional spaces that allow to deal with a possible singularity at time $0$ (that we do not have here),  but he does not explicitly provide a uniqueness result.

Observe that in \cite{Gozzi97, CerraiGozzi95} the problem is approached 
by studying the properties of the transition semigroup defined in
 (\ref{eq:defPt}) on the space $C_b(H)$ of the continuous bounded 
function (or in some cases, on the space $B_b(H)$ of bounded function)
 defined on $H$ introducing a new notion of semigroup 
(see also \cite{Priola99}). This kind of methodology 
\emph{structurally} requires the initial datum $g$ 
to belong to $C_b(H)$ (or $B_b(H)$) and then Kolmogorov equations 
with unbounded initial datum cannot be studied.
\end{rem}
\begin{rem} \label{R813}
The ideas we used here to prove the relation between strong and mild solutions of the Kolmogorov equations can be used to study second order Hamilton-Jacobi-Bellman equation related to optimal control problems driven by stochastic PDEs and provide consequently verification theorems. This kind of approach is used for example in Section 6 of \cite{RusFab}.
\end{rem}

\noindent {\bf ACKNOWLEDGEMENTS:} The research was 
supported by the ANR Project MASTERIE 2010 BLAN-0121-01.
The second named author was partially supported by the \emph{Post-Doc Research Grant} of
\emph{Unicredit \& Universities} and his research has been developed in
the framework of the center of excellence LABEX MME-DII
(ANR-11-LABX-0023-01).
The authors are grateful to two anonymous Referees for
reading carefully the paper and helping us in improving its quality.

\addcontentsline{toc}{chapter}{Bibliography}
\bibliographystyle{plain}
\bibliography{biblio7}

\end{document}